\documentclass[reqno]{amsart}
\usepackage[T1]{fontenc}
\usepackage{amssymb}
\usepackage{amsthm}
\usepackage{xcolor}
\usepackage{mathtools}
\usepackage{thmtools}
\usepackage[breaklinks,colorlinks,linkcolor=blue,citecolor=teal]{hyperref}

\newcommand{\psp}[1]{\sp{(#1)}} 
\newcommand{\norm}[2][y]{\if#1y\left\fi\lVert#2\if#1y\right\fi\rVert} 
\newcommand{\abs}[2][y]{\if#1y\left\fi\lvert#2\if#1y\right\fi\rvert} 
\newcommand{\Real}{\mathbb{R}}
\renewcommand{\natural}{\mathbb{N}} 
\newcommand{\LL}{\mathrm{L}}
\newcommand{\HH}{\mathrm{H}}
\newcommand{\CC}{\mathrm{C}}
\newcommand{\dd}[1][y]{\if#1y\,\fi{\mathrm d}} 
\renewcommand{\div}{\mathrm{div}} 
\newcommand{\poly}{\Pi}
\DeclareMathOperator{\proj}{proj}
\DeclareMathOperator{\Sym}{Sym}
\DeclareMathOperator{\Skew}{Skew}
\newcommand{\sphere}{\mathbb{S}}

\declaretheorem[name=Theorem,numberwithin=section]{thm}
\declaretheorem[name=Lemma,sibling=thm]{lemma}
\declaretheorem[name=Proposition,sibling=thm]{proposition}
\declaretheorem[name=Corollary,sibling=thm]{corollary}
\declaretheorem[name=Remark,sibling=thm,style=remark]{remark}
\declaretheorem[name=Definition,sibling=thm,style=definition]{definition}

\numberwithin{equation}{section}

\begin{document}

\title[Orthogonal polynomial projection error in Dunkl--Sobolev norms]{Orthogonal polynomial projection error in Dunkl--Sobolev norms in the ball}

\author{Gonzalo A. Benavides}
\address{CI\textsuperscript{2}MA and Departamento de Ingenier\'ia Matem\'atica, Universidad de Concepci\'on, Casilla 160-C, Concepci\'on, Chile}
\email{gobenavides@udec.cl}

\author{Leonardo E. Figueroa}
\address[Corresponding author]{CI\textsuperscript{2}MA and Departamento de Ingenier\'ia Matem\'atica, Universidad de Concepci\'on, Casilla 160-C, Concepci\'on, Chile}
\email{lfiguero@ing-mat.udec.cl}

\begin{abstract}
We study approximation properties of weighted $\LL^2$-orthogonal projectors onto spaces of polynomials of bounded degree in the Euclidean unit ball, where the weight is of the reflection-invariant form $(1-\norm{x}^2)^\alpha \prod_{i=1}^d \abs{x_i}^{\gamma_i}$, $\alpha, \gamma_1, \dotsc, \gamma_d > -1$.
Said properties are measured in Dunkl--Sobolev-type norms in which the same weighted $\LL^2$ norm is used to control all the involved differential-difference Dunkl operators, such as those appearing in the Sturm--Liouville characterization of similarly weighted $\LL^2$-orthogonal polynomials, as opposed to the partial derivatives of Sobolev-type norms.
The method of proof relies on spaces instead of bases of orthogonal polynomials, which greatly simplifies the exposition.
\end{abstract}
\maketitle



\thanks{G.A.~Benavides and L.E.~Figueroa acknowledge partial support from CONICYT-Chile through project Fondecyt Regular 1181957.}

\section{Introduction}\label{sec:introduction}

Let $B^d$ denote the unit ball of $\Real^d$, $\alpha > -1$ and let the weight function $W_\alpha \colon B^d \to \Real$ be defined by $W_\alpha(x) = (1-\norm{x}^2)^\alpha$ with $\norm{\cdot}$ being the Euclidean norm.
Let $\LL^2_\alpha$ denote the weighted Lebesgue space $\LL^2(B^d, W_\alpha) := \{ W_\alpha^{-1/2} f \mid f \in \LL^2(B^d) \}$, whose natural squared norm is $\norm{u}_\alpha^2 := \int_{B^d} \abs{u}^2 W_\alpha$.
In \cite{Figueroa:2017a} one of the authors proved that the orthogonal projector $S^\alpha_N$ mapping $\LL^2_\alpha$ onto $\poly^d_N$ (the space of $d$-variate polynomials of degree less than or equal to $N$) satisfies the bound
\begin{equation}\label{old-result}
(\forall\,u\in\HH^l_\alpha) \quad \norm{u - S^\alpha_N(u)}_{\alpha;1}
\leq C \, N^{3/2-l} \norm{u}_{\alpha;l},
\end{equation}
where $C > 0$ depends on $\alpha$ and the integer $l \geq 1$ only, and, for every integer $m \geq 1$, $\HH^\alpha_m$ denotes the weighted Sobolev space whose natural squared norm is $\norm{u}_{\alpha;m} := \sum_{k=0}^m \norm{ \nabla^k u}_\alpha^2$ (here $\nabla^k$ is the $k$-fold gradient).

The purpose of this work is proving an analogue of \eqref{old-result} for a class of reflection-invariant weights involving, fittingly, differential-difference Dunkl operators \cite[Sec.~6.4]{DunklXu:2014} instead of partial derivatives.
In order to state this analogue we introduce now the rest of the minimal necessary notation.
Given $\alpha > -1$ and $\gamma=(\gamma_1,\dotsc,\gamma_d) \in (-1,\infty)^d$, let the weight function $W_{\alpha,\gamma} \colon B^d \to \Real$ be defined by
\begin{equation*}
W_{\alpha,\gamma}(x) := (1-\norm{x}^2)^\alpha \prod_{i=1}^d \abs{x_i}^{\gamma_i}.
\end{equation*}
We denote by $\LL^2_{\alpha,\gamma}$ the weighted Lebesgue space $\LL^2(B^d,W_{\alpha,\gamma})$, whose natural inner product and squared norm are $\langle u,v\rangle_{\alpha,\gamma}:=\int_{B^d} u \, v \, W_{\alpha,\gamma}$ and $\norm{u}_{\alpha,\gamma}^2 := \int_{B^d} \abs{u}^2 W_{\alpha,\gamma}$, respectively.
Let $S\sp{\alpha,\gamma}_N$ be the orthogonal projector mapping $\LL^2_{\alpha,\gamma}$ onto $\poly^d_N$.
For $j \in \{1, \dotsc, d\}$ the Dunkl operator $\mathcal{D}\psp{\gamma}_j$ is defined by
\begin{equation*}
\mathcal{D}\psp{\gamma}_j u(x) := \partial_j u(x) + \frac{\gamma_j}{2} \left( u(x) - u(x_1, \dotsc, \overbrace{-x_j}^{\smash{\text{$j$-th entry}}}, \dotsc, x_d) \right).
\end{equation*}
Given an integer $m \geq 0$, we define the \emph{Dunkl--Sobolev space} $\HH^m_{\alpha,\gamma}$ as the topological completion of $\CC^m(\overline{B^d})$ with respect to the norm $\norm{u}_{\alpha,\gamma;m} := \sum_{k=0}^m \norm{ (\mathcal{D}\psp{\gamma})^k u}_{\alpha,\gamma}^2$, where $(\mathcal{D}\psp{\gamma})^k$ is the $k$-fold Dunkl gradient constructed in terms of the Dunkl operators (we reintroduce the Dunkl operators and Dunkl--Sobolev spaces in their proper context in \eqref{parameter-raising-operator} and \autoref{def:DunklSobolevSpace}, respectively).
Our main result is
\begin{thm}\label{thm:main}
For all integers $1 \leq r \leq l$, $\alpha \in (-1,\infty)$ and $\gamma \in (-1,\infty)^d$, there exists $C = C(\alpha,\gamma,l,r) > 0$ such that
\begin{equation*}
(\forall \, u \in \HH^l_{\alpha,\gamma}) \quad \norm{u - S\sp{\alpha,\gamma}_N (u)}_{\alpha,\gamma;r} \leq C \, N^{-1/2 + 2r - l} \norm{u}_{\alpha,\gamma;l}.
\end{equation*}
\end{thm}

This work builds upon a lineage of works which proved results analogous to \autoref{thm:main}, all of which correspond, in our notation, to cases with $\gamma = 0$, so the involved weights lack interior singularities and the Dunkl operators reduce to partial derivatives.
In \cite[Th.~2.2 and Th.~2.4]{CQ:1982} our main result was proved in dimension $d = 1$ when the $\alpha = -1/2$ (Chebyshev case) and when $\alpha = 0$ (Legendre case); see also the streamlined proofs for these cases at \cite[Ch.~5]{CHQZ-I}.
In \cite[Th,~2.6]{Guo:2000a}, the one-dimensional case was proved for general $\alpha$ (Gegenbauer case).
In \cite[Th.~2.6]{Xu:2018}, the one-dimensional case with general asymmetric $(1-x)^\alpha(1+x)^\beta$ weight (Jacobi case) was proved.
In \cite[Th.~3.11]{Figueroa:2017b}, \autoref{thm:main} was extended to dimension $d = 2$ for general $\alpha$.
Finally, in \cite[Th.~1.1]{Figueroa:2017a}, a new technique of proof, based on orthogonal polynomial \emph{spaces} instead of orthogonal polynomial \emph{bases} (thus circumventing the need for spectral differentiation formulas, which by \cite{Figueroa:2017b} had made the necessary algebraic manipulation very long in comparison) allowed for extending the result to arbitrary dimension for general $\alpha$.

In the $\gamma = 0$ cases cited above, the analogues of \autoref{thm:main} are results of provably non-optimal polynomial approximation with respect to the power on $N$, caused by the mismatch between the orthogonality that defines the projection operator $S\sp{\alpha,\gamma}_N$ ---which can be characterized as a generalized Fourier series truncation operator; cf.\ \eqref{spanning-consequence}--- and the Hilbert norm in which the error is measured (see the references provided in \cite[Sec.~1]{Figueroa:2017a} for optimal polynomial approximation results).
The same mismatch occurs in this work, so we expect \autoref{thm:main} to be non-optimal too; however, we cannot be sure because we are not aware of best approximation results for the general $\gamma$ case.

The study of approximation results involving weights such as $W_{\alpha,\gamma}$ is interesting, first, as an archetype of weights of interior singularities, as its highly symmetric form allows for sourcing useful results from the theory of reflection-invariant orthogonal polynomials \cite[Ch.~6 and Ch.~7]{DunklXu:2014}.
Secondly, there is an intimate connection between orthogonal polynomials in the ball with respect to $W_{\alpha,\gamma}$ and orthogonal polynomials in the simplex with respect to weights that are products of powers of signed distances to their faces \cite[Subsec.~8.2]{DunklXu:2014}; as this reference attests, when mapping orthogonal polynomials in the ball to orthogonal polynomials in the simplex, only the fully reflection-symmetric of the former participate, and for these the Dunkl operators reduce to partial derivatives.
Thirdly, we fully expect that the techniques and partial results introduced below, not least the adaptation of the orthogonal polynomials spaces-based techniques of \cite{Figueroa:2017a} to the present situation, will prove useful again in the pursuit of further approximation results.

Our main result involves weighted Dunkl--Sobolev spaces instead of the better understood weighted Sobolev spaces because it is in terms of the former that the contours of the argument in \cite{Figueroa:2017a} can be reproduced.
This is readily apparent because the characterization of $\LL^2_{\alpha,\gamma}$-orthogonal polynomials as eigenfunctions of Sturm--Liouville-type operators occurs in terms of Dunkl operators \cite[Th.~8.1.3]{DunklXu:2014}; said characterization is essential for our way of inferring approximation rates out of the regularity of the function being approximated.

We remark that $r = 0$ (i.e., approximation error measured in $\LL^2_{\alpha,\gamma}$) lies outside of the scope of \autoref{thm:main};
indeed, in this case, the provably optimal power on $N$ is $-l$ (cf.\ \autoref{cor:L2-truncation-error} below), outside of the pattern set by our main result.

The outline of this article is as follows.
We finish this introductory \autoref{sec:introduction} introducing some additional notation.
In \autoref{sec:Dunkl} we introduce the reflections, Dunkl operators and main Dunkl--Sobolev spaces that participate in this work.
In \autoref{sec:OPS} we introduce orthogonal polynomials spaces and their interaction with Dunkl operators and certain generalizations thereof.
In \autoref{sec:SL-approx} we put the differential-difference Sturm--Liouville operator the abovementioned orthogonal polynomial spaces are eigenspaces in a suitable weak form, prove preliminary approximation results upon it and prove our main result.
At last, in \autoref{sec:sharpness} we prove the sharpness of our main result for special values of its Dunkl--Sobolev regularity parameters and give a brief conclusion.

Given, $j \in \{1, \dotsc, d\}$, let $e_j \in \Real^d$ be Cartesian unit vector in the $j$-th direction; i.e., $(e_j)_i$ is $1$ if $i = j$ and $0$ otherwise.
We will denote the Euclidean norm by $\norm{\cdot}$.
We will denote the space of $d$-variate polynomials by $\poly^d$; we have already introduced its subspace $\poly^d_N$ consisting of polynomials of total degree less than or equal to $N$.
We will adopt the convention that, for $N < 0$, $\poly^d_N = \{ 0 \}$.

Given an open $\Omega \subset \Real^d$ we will denote the integral of functions $f \colon \Omega \to \Real$ with respect to the Lebesgue measure simply by $\int_\Omega f(x) \dd x$.
We will denote by $\sigma_{d-1}$ the surface measure of $\sphere^{d-1}$, the unit sphere of $\Real^d$ \cite[Ex.~3.10.82]{Bogachev}.
For all Lebesgue-integrable $f$,
\begin{equation}\label{polarIntegral}
\int_{\Real^d} f(x) \dd x = \int_0^\infty \int_{\sphere^{d-1}} f(r y) \, r^{d-1} \dd\sigma_{d-1}(y) \dd r.
\end{equation}

We denote by $\natural$ the set of strictly positive integers and $\natural_0 := \{0\} \cup \natural$.
Members of $[\natural_0]^d$ will be called multi-indices and for every multi-index $a \in [\natural_0]^d$, point $x \in \Real^d$ and regular enough real-valued function $f$ defined on some open set of $\Real^d$ we shall write $\abs{a} = \sum_{i=1}^d a_i$, $x^a = \prod_{i=1}^d x_i^{a_i}$ and $\partial_a f = \partial^{\abs{a}} f/(\partial x_1^{a_1} \dotsm \partial x_d^{a_d})$.

Setting $a_i = 1$, $p_i = 2$, $t_1 = 0$, $t_2 = 1$, $\alpha_i = \gamma_i + 1$ and $f(u) = (1-u)^\alpha$ in \cite[Th.~1.8.5]{AAR:1999} it readily follows that
\begin{equation}\label{one-norm}
\int_{B^d} W_{\alpha,\gamma}(x) \dd x = \frac{\prod_{i=1}^d \Gamma\left( \frac{\gamma_i+1}{2} \right)}{\Gamma\left( \frac{d+\sum_{i=1}^d \gamma_i}{2} \right)} \, \mathrm{B}\left(\frac{d}{2} + \frac{1}{2} \sum_{i=1}^d \gamma_i , \alpha+1 \right),
\end{equation}
where $\Gamma$ and $\mathrm{B}$ denote Gamma and Beta functions respectively;
these functions being finite for positive arguments, it follows that the constraints $\alpha > -1$ and $\gamma \in (-1,\infty)^d$ are precisely those that ensure that the above integral is finite.
As a consequence of \eqref{one-norm}, $\LL^\infty(B^d) \subset \LL^2_{\alpha,\gamma}$.
In particular, every polynomial, being a bounded function in $B^d$, belongs to $\LL^2_{\alpha,\gamma}$.

We finish this introductory section noting that we mostly omit the dimension $d$ from the notation of e.g., function spaces, in order to avoid cluttering and because all of our arguments work independently of the dimension.

\section{Dunkl operators and weighted Dunkl--Sobolev spaces}\label{sec:Dunkl}

Given $j \in \{1, \dotsc, d\}$ let $\sigma_j \colon B^d \to B^d$ be the reflection defined by
\begin{equation}\label{reflection}
(\forall \, x \in B^d) \quad \sigma_j x := x - 2 \, x_j \, e_j;
\end{equation}
that is, $\sigma_j$ flips the sign of the $j$-th component of its argument.
The group generated by $\{\sigma_j \mid 1 \leq j \leq d\}$ with function composition as the group operation is (isomorphic to) the Coxeter group $\mathbb{Z}_2^d$ \cite[Sec.~7.5]{DunklXu:2014}.
Given a scalar-, vector- or tensor-valued function $f$ on $B^d$, we shall write $\sigma_j^* := f \circ \sigma_j$.
We will say that $f$ is $\sigma_j$-even (resp.\ $\sigma_j$-odd) if $\sigma_j^* f = f$ (resp.\ $\sigma_j^* f = -f$) almost everywhere.
On defining
\begin{equation}\label{sym-and-skew}
\Sym_j(f) := \frac{f + \sigma_j^* f}{2} \qquad\text{and}\qquad \Skew_j(f) := \frac{f - \sigma_j^* f}{2},
\end{equation}
every such $f$ admits
\begin{equation}\label{parity-decomposition}
f = \Sym_j(f) + \Sym_j(f)
\end{equation}
as its unique decomposition into a $\sigma_j$-even and a $\sigma_j$-odd part.
For every $i, j \in \{1, \dotsc, d\}$, $\sigma_i$ and $\sigma_j$ commute.
Therefore, so do the operator pairs $(\sigma_i^*, \sigma_j^*)$, $(\Sym_i, \Sym_j)$ and $(\Sym_i, \Skew_j)$.
It follows that
\begin{equation}\label{twofold-parity-decomposition}
f = \Sym_i(\Sym_j(f)) + \Sym_i(\Skew_j(f)) + \Skew_i(\Sym_j(f)) + \Skew_i(\Skew_j(f)).
\end{equation}
is the only decomposition of $f$ into all four combinations of $\sigma_i$- and $\sigma_j$-parity.
Following \cite[Def.~6.4.4]{DunklXu:2014}, we further introduce the operators $\rho_j$ by
\begin{equation}\label{rho}
\rho_j(f)(x) := \frac{f(x) - f(\sigma_j x)}{x_j}
= \frac{2 \, \Skew_j(f)(x)}{x_j},
\end{equation}
where, whenever $x_j = 0$, the ratio must be interpreted as the corresponding limit; namely, $2 \, \partial_j f(x)$.
The following variant of Hadamard's lemma (cf.\ \cite[Sec.~3.20]{Petrovski:1966}) encapsulates the properties of the $\rho_j$ operators we shall need later.

\begin{proposition}\label{pro:Hadamard}
Let $j \in \{1, \dotsc, d\}$ and $f \in \CC^r(\overline{B^d})$, $r \geq 1$.
Then, $\rho_j(f) \in \CC^{r-1}(\overline{B^d})$ and, for all multi-indices $a$ with $0 \leq \abs{a} \leq r-1$,
\begin{equation}\label{HadamardFactorBound}
\norm{\partial_a \rho_j(f)}_{\infty} \leq 2 \norm{\partial_a \partial_j f}_{\infty}.
\end{equation}
If $f$ happens to be a polynomial of degree $n$, $\rho_j(f)$ is also a polynomial of degree at most $n-1$.
\begin{proof}
Throughout this proof, for all $z \in \overline{B^d}$ we set $z' = (z_1, \dotsc, z_{d-1})$ and $z'' = (z_1, \dotsc, z_{d-2})$ so that $z = (z', z_d) = (z'', z_{d-1}, z_d)$.
Also, given a function $h \colon \overline{B^d} \to \Real$ we denote its modulus of continuity by $\omega(\,\cdot\,;h)$; that is, for all $t \in [0, \infty]$, $\omega(t;h) := \sup\left\{ \abs{h(x) - h(y)} \mid x,y \in \overline{B^d}, \ \abs{x-y}\leq t \right\}$.
We also assume, without loss of generality, that $j = d$.

Given $k \in \natural_0$ let the integral operator $H_k$ be defined by
\begin{equation}\label{HadamardAuxiliaryOperator}
H_k(h)(x) := \int_{-1}^1 s^k \, h(x', s\,x_d) \dd s.
\end{equation}
First, let us note that
\begin{equation}\label{proto-Hadamard-continuity}
(\forall\,h \in \CC(\overline{B^d})) \quad H_k(h) \in \CC(\overline{B^d}).
\end{equation}
Indeed, let $h \in \CC(\overline{B^d})$.
Then, for all $x, y \in \overline{B^d}$,
\begin{equation*}
\abs{H_k(h)(x) - H_k(h)(y)}
\leq \int_{-1}^1 \abs{ h(x', s\,x_d) - h(y', s\,x_d) } \dd s
\leq 2 \, \omega(\abs{x-y}; h).
\end{equation*}
Thus, $0 \leq \omega(\,\cdot\,;H_k(h)) \leq 2 \, \omega(\,\cdot\,;h)$ so $H_k(h)$ inherits the uniform continuity of $h$, which, in turn, stems from the fact that $\overline{B^d}$ is compact.
Also, directly from the definition \eqref{HadamardAuxiliaryOperator},
\begin{equation}\label{HAOONB}
(\forall\,h \in \CC(\overline{B^d})) \quad \norm{H_k(h)}_\infty \leq 2 \norm{h}_\infty.
\end{equation}
Next, we note that, as a consequence of the Fundamental Theorem of Calculus, for all $i \in \{1, \dotsc, d\}$,
\begin{equation}\label{FDUC}
(\forall\,h\in\CC^1(\overline{B^d})) \quad \abs{\frac{h(x+\eta\,e_i) - h(x)}{\eta} - \partial_i h(x)} \leq \omega(\abs{\eta}; \partial_i h).
\end{equation}
Further, we affirm that
\begin{equation}\label{proto-Hadamard-differentiability}
(\forall\,h \in \CC^1(\overline{B^d})) \quad \partial_i H_k(h)
= \begin{cases} H_k(\partial_i h) & \text{if } i \neq d,\\ H_{k+1}(\partial_d h) & \text{if } i = d. \end{cases}
\end{equation}
Indeed, let $h \in \CC^1(\overline{B^d})$.
Let $i \in \{1, \dotsc, d-1\}$; without loss of generality we can assume that $i = d-1$.
Then,
\begin{multline*}
\abs{\frac{H_k(h)(x+\eta e_{d-1}) - H_k(h)(x)}{\eta} - H_k(\partial_{d-1} h)(x)}\\
\leq \int_{-1}^1 \abs{\frac{h(x'', x_{d-1} + \eta, s\,x_d) - h(x'', x_{d-1}, s\,x_d)}{\eta} - \partial_{d-1} h(x', s\,x_d)} \dd s
\xrightarrow{\eta \to 0} 0
\end{multline*}
because, per \eqref{FDUC}, the last integrand tends to $0$ as $\eta$ tends to $0$ uniformly with respect to $s$.
If $i = d$,
\begin{multline*}
\abs{\frac{H_k(h)(x+\eta e_d) - H_k(h)(x)}{\eta} - H_{k+1}(\partial_d h)(x)}\\
\leq \int_{-1}^1 \abs{\frac{h(x', s(x_d + \eta)) - h(x', s\,x_d)}{\eta} - s \, \partial_d h(x', s\,x_d)} \dd s\\
\leq \int_{-1}^1 \abs{\frac{h(x', s \, x_d + s\, \eta) - h(x', s\,x_d)}{s\,\eta} - \partial_d h(x', s\,x_d)} \dd s \xrightarrow{\eta \to 0} 0,
\end{multline*}
again by \eqref{FDUC} and the fact that $\abs{s\,\eta} \leq \abs{\eta}$.
Thus we have justified \eqref{proto-Hadamard-differentiability}.

Let $f \in \CC^r(\overline{B^d})$.
Then, $\rho_d(f) = H_0(\partial_d f)$.
Indeed, if $x_d = 0$, $\rho_d(f)(x) = 2\,\partial_d f(x)$ and $H_0(\partial_d f)(x)$ obviously coincide.
If $x_d \neq 0$, by the Fundamental Theorem of Calculus and the definition in \eqref{rho},
\begin{equation*}
\rho_d(f)(x) = \frac{1}{x_d} \int_{-x_d}^{x_d} \partial_d f(x', t) \dd t
= \int_{-1}^1 \partial_d f(x', s \, x_d) \dd s
= H_0(\partial_d f)(x).
\end{equation*}
With $\rho_d f$ characterized in this way, its membership in $\CC^{r-1}(\overline{B^d})$ and the bound \eqref{HadamardFactorBound} stem from \eqref{proto-Hadamard-continuity}, \eqref{HAOONB} and \eqref{proto-Hadamard-differentiability}.

Let us note that if $h$ happens to be the monomial $h(x) = \prod_{i=1}^d x_i^{\alpha_i}$, $\alpha_1, \dotsc, \alpha_d \in \natural_0$, a direct computation shows that $H_0(h) = \frac{1-(-1)^{\alpha_d+1}}{\alpha_d+1} h$.
Thus, $H_0$ maps polynomials to polynomials of at most the same total degree.
Hence, if $f$ is a polynomial of total degree $n$, $\rho_d(f) = H_0(\partial_d f)$ is a polynomial of total degree at most $n-1$.
\end{proof}
\end{proposition}

Given any $\gamma \in \Real^d$, the map that to each $e_j$ and $-e_j$, $j \in \{1, \dotsc, d\}$ associates $\gamma_j$ is $\mathbb{Z}_2^d$ invariant, so it is a multiplicity function in the sense of \cite[Def.~6.4.1]{DunklXu:2014}.
The Dunkl operators associated with (the multiplicity function induced by) $\gamma$ \cite[Def.~6.4.2]{DunklXu:2014} are
\begin{equation}\label{parameter-raising-operator}
(\forall\,j\in\{1,\dotsc,d\}) \quad \mathcal{D}\psp{\gamma}_j q(x)
:= \partial_j q(x) + \frac{\gamma_j}{2} \rho_j(q)(x)
\stackrel{\eqref{rho}}{=} \partial_j q(x)+\frac{\gamma_j}{2} \frac{q(x)-q(\sigma_j x)}{x_j}.
\end{equation}
Through \autoref{pro:Hadamard} the Dunkl operators inherit from the standard partial derivatives the inclusions
\begin{equation}\label{Dunkl-inclusions}
\mathcal{D}\psp{\gamma}_j \left( \CC^m(\overline{B^d}) \right) \subseteq \CC^{m-1}(\overline{B^d})
\quad\text{and}\quad
\mathcal{D}\psp{\gamma}_j \left( \poly^d_m \right) \subseteq \poly^d_{m-1}
\end{equation}
for $m \in \natural$ and $m \in \natural_0$, respectively.

The following commutation relations are particularizations of Prop.~6.4.3, Th.~6.4.9 and Prop.~6.4.10 \cite{DunklXu:2014}, respectively:
\begin{equation}\label{Dunkl-reflection-commutator}
\mathcal{D}\psp{\gamma}_j \sigma_i^* = \begin{cases} \sigma_i^* \mathcal{D}\psp{\gamma}_j & \text{if } i \neq j,\\ -\sigma_j^* \mathcal{D}\psp{\gamma}_j & \text{if } i = j, \end{cases}
\end{equation}
\begin{equation}\label{Dunkl-commutator}
\mathcal{D}\psp{\gamma}_i \mathcal{D}\psp{\gamma}_j = \mathcal{D}\psp{\gamma}_j \mathcal{D}\psp{\gamma}_i,
\end{equation}
\begin{equation}\label{Dunkl-xi-commutator}
\mathcal{D}\psp{\gamma}_j(x_i q) = \begin{cases} x_i \mathcal{D}\psp{\gamma}_j q & \text{if } i \neq j,\\ x_j \mathcal{D}\psp{\gamma}_j q + q + \gamma_j \sigma_j^* q & \text{if } i = j.\end{cases}
\end{equation}
Note that in \eqref{Dunkl-xi-commutator} and in the sequel we commit the common abuse of notation of denoting maps of the form $x \mapsto x_i q(x)$ simply as $x_i q$.
Some consequences of \eqref{Dunkl-reflection-commutator} are
\begin{equation}\label{Dunkl-sym-skew}
\mathcal{D}\psp{\gamma}_j \Sym_i = \begin{cases} \Sym_i \mathcal{D}\psp{\gamma}_j & \text{if } i \neq j, \\ \Skew_i \mathcal{D}\psp{\gamma}_j & \text{if } i = j \end{cases}
\quad\text{and}\quad
\mathcal{D}\psp{\gamma}_j \Skew_i = \begin{cases} \Skew_i \mathcal{D}\psp{\gamma}_j & \text{if } i \neq j, \\ \Sym_i \mathcal{D}\psp{\gamma}_j & \text{if } i = j. \end{cases}
\end{equation}
Also, as
\begin{equation}\label{reflection-xi-commutator}
x_j \sigma_i^* q = \begin{cases} \sigma_i^*(x_j q) & \text{if } i \neq j,\\ -\sigma_j^*(x_j q) & \text{if } i = j, \end{cases}
\end{equation}
we further have
\begin{equation}\label{xi-sym-skew}
x_j \Sym_i q = \begin{cases} \Sym_i(x_j q) & \text{if } i \neq j, \\ \Skew_i(x_j q) & \text{if } i = j \end{cases}
\quad\text{and}\quad
x_j \Skew_i q = \begin{cases} \Skew_i(x_j q) & \text{if } i \neq j, \\ \Sym_i(x_j q) & \text{if } i = j. \end{cases}
\end{equation}

Because of the commutation property \eqref{Dunkl-commutator}, we can unambiguously use the multi-index notation to express compositions of Dunkl operators; hence, given $a \in [\natural_0]^d$, we shall write $\mathcal{D}\psp{\gamma}_a := (\mathcal{D}\psp{\gamma}_1)^{a_1} \circ \dotsm \circ (\mathcal{D}\psp{\gamma}_d)^{a_d}$.
We can now compactly express the following consequence of \autoref{pro:Hadamard}:
For all multi-indices $a \in [\natural_0]^d$ and $f \in \CC^{\abs{a}}(\overline{B^d})$,
\begin{equation}\label{Dunkl-bounded-by-derivatives}
\norm{D\psp{\gamma}_a f}_\infty \leq \prod_{i=1}^d (1+\abs{\gamma_i})^{a_i} \, \norm{\partial_a f}_\infty.
\end{equation}

We define the Dunkl gradient by $\mathcal{D}\psp{\gamma} f := \sum_{j=1}^d \mathcal{D}\psp{\gamma}_j(f) \, e_j$.
Given $m \in \natural_0$ we define the Sobolev-type inner product $\langle \cdot, \cdot \rangle_{\alpha,\gamma;m} \colon \CC^m(\overline{B^d}) \times \CC^m(\overline{B^d}) \to \Real$ by
\begin{equation}\label{DunklSobolevInnerProduct}
(\forall \, p,q \in \CC^m(\overline{B^d})) \quad \langle p, q \rangle_{\alpha,\gamma;m}
:= \sum_{k=0}^m \left\langle (\mathcal{D}\psp{\gamma})^k p, (\mathcal{D}\psp{\gamma})^k q \right\rangle_{\alpha,\gamma},
\end{equation}
where $(\mathcal{D}\psp{\gamma})^k$ is the $k$-fold Dunkl gradient.
Using the multi-index notation, this inner product can also be expressed as $(p,q) \mapsto \sum_{k=0}^m \sum_{\abs{a}=k} \binom{k}{a} \langle \mathcal{D}\psp{\gamma}_a p, \mathcal{D}\psp{\gamma}_a q \rangle_{\alpha,\gamma}$ (here $\binom{k}{a} = \frac{k!}{a_1! \dotsm a_d!}$ is the number of times $\mathcal{D}\psp{\gamma}_a p$ with $\abs{a} = k$ appears in the $k$-dimensional array-valued $(\mathcal{D}\psp{\gamma})^k p$) and is of course bounded from above and below by positive-constant multiples of $(p,q) \mapsto \sum_{\abs{a} \leq m} \langle \mathcal{D}\psp{\gamma}_a p, \mathcal{D}\psp{\gamma}_a q \rangle_{\alpha,\gamma}$.

We define now in some detail the function spaces involved in our main result \autoref{thm:main}.

\begin{definition}\label{def:DunklSobolevSpace}
Given $m \in \natural_0$, we define $\HH^m_{\alpha,\gamma}$ as the topological completion of $(\CC^m(\overline{B^d}), \norm{\cdot}_{\alpha,\gamma;m})$.

That is, up to isometry, $\HH^m_{\alpha,\gamma}$ is the space of equivalence classes of Cauchy sequences of $(\CC^m(\overline{B^d}), \norm{\cdot}_{\alpha,\gamma;m})$ with respect to the equivalence relation $\sim$ defined by $(x_n)_{n\in\natural} \sim (y_n)_{n\in\natural} \iff \lim_{n\to\infty} \norm{x_n-y_n}_{\alpha,\gamma;m} = 0$, equipped with the metric $(x,y) \mapsto \lim_{n \to \infty} \norm{x_n - y_n}_{\alpha,\gamma;m}$, where $(x_n)_{n\in\natural}$ and $(y_n)_{n\in\natural}$ are any representatives of the equivalence classes $x$ and $y$, respectively, which makes it a complete metric space.
Identifying each $f \in \CC^m(\overline{B^d})$ with the equivalence class of the constant sequence $(f)_{n\in\natural}$, $\CC^m(\overline{B^d})$ is a dense subset of $\HH^m_{\alpha,\gamma}$ \cite[Th.~III.33.VII]{Kuratowski:1966}, \cite[Th.~4.3.19]{Engelking:1989}.

It is easily checked that the map $(x,y) \mapsto \lim_{n \to \infty} \langle x_n, y_n \rangle_{\alpha,\gamma;m}$, where again $(x_n)_{n\in\natural}$ and $(y_n)_{n\in\natural}$ are any representatives of the equivalence classes $x$ and $y$, respectively, is a well defined inner product that induces the above metric, whence $\HH^m_{\alpha,\gamma}$ is a Hilbert space.
We denote that inner product by $\langle \cdot, \cdot \rangle_{\alpha,\gamma;m}$ as well.
\end{definition}

\begin{proposition}\label{pro:polynomials-dense}
Polynomials are dense in $\HH^m_{\alpha,\gamma}$.
\begin{proof}
Let $f \in \HH^m_{\alpha,\gamma}$ and $\epsilon > 0$.
By the characterization of $\HH^m_{\alpha,\gamma}$ as a topological completion in \autoref{def:DunklSobolevSpace}, there exists $g \in \CC^m(\overline{B^d})$ such that $\norm{f-g}_{\alpha,\gamma;m} < \epsilon/2$.
Now, $g$ can be extended to a $\CC^m(\Real^d)$ function $\tilde g$ \cite{Whitney:1934d}, which, by smooth truncation if necessary, can be assumed to have its support contained in the ball $B(0,2)$.
By \cite[Cor.~3]{EvardJafari:1994d}, there exists a polynomial $p$ such that
\begin{equation*}
\sum_{\abs{a} \leq m} \sup_{B^d} \abs{\partial_a g - \partial_a p}
= \sum_{\abs{a} \leq m} \sup_{B^d} \abs{\partial_a \tilde g - \partial_a p} < \frac{\epsilon}{2 \, c_{d,m}},
\end{equation*}
where $c_{d,m} = \norm{1}_{\alpha,\gamma}^{1/2} \, \max_{\abs{\alpha} \leq m} \binom{\abs{a}}{a}^{1/2} \, \max_{\abs{\alpha} \leq m} \prod_{i=1}^d (1+\abs{\gamma_i})^{\alpha_i}$ (this constant is finite on account of \eqref{one-norm}).
Thus, by \eqref{Dunkl-bounded-by-derivatives} and the definition \eqref{DunklSobolevInnerProduct},
\begin{multline*}
	\norm{g-p}_{\alpha,\gamma;m}
	\leq \norm{1}_{\alpha,\gamma}^{1/2} \max_{\abs{\alpha} \leq m} \binom{\abs{a}}{a}^{1/2} \left( \sum_{\abs{\alpha} \leq m} \norm{D\psp{\gamma}_a g - D\psp{\gamma}_a p}_\infty^2 \right)^{1/2}\\
	\leq \norm{1}_{\alpha,\gamma}^{1/2} \max_{\abs{\alpha} \leq m} \binom{\abs{a}}{a}^{1/2} \max_{\abs{\alpha} \leq m} \prod_{i=1}^d (1+\abs{\gamma_i})^{\alpha_i} \sum_{\abs{\alpha} \leq m} \norm{\partial_a g - \partial_a p}_\infty
	< \frac{\epsilon}{2}.
\end{multline*}
\end{proof}
\end{proposition}

\begin{remark}\label{rem:no-distributional-Dunkl}
We define our Dunkl--Sobolev spaces as topological completions of strongly differentiable functions with respect to the chosen norm; that is, `H' spaces in the nomenclature of Meyers \& Serrin \cite{MS}.
One might also define Dunkl--Sobolev spaces intrinsically, as spaces of (classes of equivalence of) $\LL^2_{\alpha,\gamma}$ functions whose Dunkl operators up to a certain order still belong to $\LL^2_{\alpha,\gamma}$; i.e., `W' spaces in the nomenclature of \cite{MS}.
To the latter end distributional generalizations of the Dunkl operators (see, e.g., \cite[Th.~4.4]{Trimeche:2001}) might be required to properly define their action on non-differentiable functions.
However, we do not know if such `W' spaces would be appropriate substitutes for (perhaps even identical to) our `H' spaces.
\end{remark}

\section{Orthogonal polynomial spaces}\label{sec:OPS}

Let $\mathcal{V}\sp{\alpha,\gamma}_k$ be the space of orthogonal polynomials of degree $k$ with respect to the weight $W_{\alpha,\gamma}$; i.e.,
\begin{equation}\label{OPS}
\mathcal{V}\sp{\alpha,\gamma}_k := \left\{ p \in \poly^d_k \mid (\forall\,q \in \poly^d_{k-1})\ \langle p, q\rangle_{\alpha,\gamma} = 0 \right\}.
\end{equation}

If $k < 0$ we adopt the convention $\poly^d_k = \{0\}$ and so $\mathcal{V}\sp{\alpha,\gamma}_k = \{0\}$.
As $W_{\alpha,\gamma}$ is centrally symmetric, it transpires from \cite[Th.~3.3.11]{DunklXu:2014} that, for all $k \in \natural_0 = \{0, 1, 2, \dotsc\}$, there holds the following parity relation:
\begin{equation}\label{parity}
(\forall \, p_k \in \mathcal{V}\sp{\alpha,\gamma}_k)\ (\forall \, x \in B^d) \quad p_k(-x) = (-1)^k p_k(x).
\end{equation}
Let $\proj\sp{\alpha,\gamma}_k$ denote the orthogonal projection from $\LL^2_{\alpha,\gamma}$ onto $\mathcal{V}\sp{\alpha,\gamma}_k$.
From \cite[Th.~3.2.18]{DunklXu:2014}, $\poly^d_n = \bigoplus_{k=0}^n \mathcal{V}\sp{\alpha,\gamma}_k$ and $\LL^2_{\alpha,\gamma} = \bigoplus_{k=0}^\infty \mathcal{V}\sp{\alpha,\gamma}_k$, whence
\begin{equation}\label{spanning-consequence}
(\forall\,n\in\natural_0) \quad S\sp{\alpha,\gamma}_n = \sum_{k=0}^n \proj\sp{\alpha,\gamma}_k
\quad\text{and}\quad
(\forall\,u\in \LL^2_{\alpha,\gamma}) \ u = \sum_{k=0}^\infty \proj\sp{\alpha,\gamma}_k(u).
\end{equation}
We mention in passing that we will denote the entrywise application of $S\sp{\alpha,\gamma}_n$ to $\LL^2_{\alpha,\gamma}$ vectors and higher-order tensors by $S\sp{\alpha,\gamma}_n$ as well.
Parseval's identity takes the form
\begin{equation}\label{parseval-L2}
\left( \forall \, u \in \LL_{\alpha,\gamma}^2 \right) \quad \norm{u}_{\alpha,\gamma}^2 = \sum_{k = 0}^\infty \norm{ \proj\sp{\alpha,\gamma}_k (u) }_{\alpha,\gamma}^2.
\end{equation}

The following proposition, analogous to \cite[Prop.~3.1]{Figueroa:2017a}, collects relations between orthogonal polynomial spaces and projectors onto them that do not involve Dunkl operators.

\begin{proposition}\label{pro:id-shift}
Let $\alpha \in (-1,\infty)$ and $\gamma \in (-1,\infty)^d$.
\begin{enumerate}
\item\label{it:weighted-id-downshift} Let $p_k \in \mathcal{V}\sp{\alpha+1,\gamma}_k$.
Then, $(1-\norm{\cdot}^2) p_k \in \mathcal{V}\sp{\alpha,\gamma}_k \oplus \mathcal{V}\sp{\alpha,\gamma}_{k+2}$.
\item\label{it:unnamed} Let $q_k \in \mathcal{V}\sp{\alpha,\gamma}_k$.
Then, $q_k = \proj\sp{\alpha+1,\gamma}_{k-2}(q_k) + \proj\sp{\alpha+1,\gamma}_k(q_k)$.
\item\label{it:proto-id-shift} Let $u \in \LL^2_{\alpha,\gamma}$.
Then, $\proj\sp{\alpha+1,\gamma}_k(u) = \proj\sp{\alpha+1,\gamma}_k\left( \proj\sp{\alpha,\gamma}_k(u) + \proj\sp{\alpha,\gamma}_{k+2}(u) \right)$.
\item\label{it:id-shift} Let $u \in \LL^2_{\alpha,\gamma}$.
Then,
\begin{equation*}
\proj\sp{\alpha+1,\gamma}_k(u)
= \proj\sp{\alpha,\gamma}_k(u) + \proj\sp{\alpha+1,\gamma}_k \circ \proj\sp{\alpha,\gamma}_{k+2}(u) - \proj\sp{\alpha+1,\gamma}_{k-2} \circ \proj\sp{\alpha,\gamma}_k(u).
\end{equation*}
\end{enumerate}
\begin{proof}
Given $q \in \poly^d_{k-1}$, $\langle (1-\norm{\cdot}^2) p_k, q\rangle_{\alpha,\gamma} = \langle p_k , q \rangle_{\alpha+1,\gamma} = 0$ by definition \eqref{OPS}.
Also, by the parity relation \eqref{parity}, $(1-\norm{\cdot}^2) p_k \perp_{\alpha,\gamma} \mathcal{V}\sp{\alpha,\gamma}_{k+1}$.
Therefore part \ref{it:weighted-id-downshift} stems from \eqref{spanning-consequence}.
An analogous argument accounts for part \ref{it:unnamed}.
Part \ref{it:proto-id-shift} comes from the fact that given $p_k \in \mathcal{V}\sp{\alpha+1,\gamma}_k$,
\begin{multline*}
\langle \proj\sp{\alpha+1,\gamma}_k(u),p_k \rangle_{\alpha+1,\gamma} = \langle u , p_k \rangle_{\alpha+1,\gamma} = \langle u , (1-\norm{\cdot}^2) p_k \rangle_{\alpha,\gamma}\\
\stackrel{\text{\ref{it:weighted-id-downshift}}}{=} \langle \proj\sp{\alpha,\gamma}_k(u) + \proj\sp{\alpha,\gamma}_{k+2}(u), (1-\norm{\cdot}^2)p_k \rangle_{\alpha,\gamma} = \langle \proj\sp{\alpha,\gamma}_k(u) + \proj\sp{\alpha,\gamma}_{k+2}(u), p_k \rangle_{\alpha+1,\gamma}.
\end{multline*}
Part \ref{it:id-shift} is obtained from adding and substracting $\proj\sp{\alpha+1,\gamma}_{k-2}(\proj\sp{\alpha,\gamma}_k(u))$ to the right hand side of part \ref{it:proto-id-shift} and using part \ref{it:unnamed}.
\end{proof}
\end{proposition}

\begin{proposition}\label{pro:flip}
Let $\alpha \in (-1,\infty)$ and $\gamma \in (-1,\infty)^d$.
\begin{enumerate}
\item\label{it:odd-has-zero-mean} Let $f \in \LL^2_{\alpha,\gamma}$ be $\sigma_j$-odd.
Then, $\int_{B^d} f(x) \, W_{\alpha,\gamma}(x) \dd x = 0$.
\item\label{it:V-flip-invariance} Given $k \in \natural_0$, $j \in \{1, \dotsc, d\}$ and $p_k \in \mathcal{V}\sp{\alpha,\gamma}_k$, $p_k \circ \sigma_j \in \mathcal{V}\sp{\alpha,\gamma}_k$ as well.
\end{enumerate}
\begin{proof}
Because of the invariance of the Lebesgue measure with respect to reflections, $\int_{B^d} f(x) \, W_{\alpha,\gamma}(x) \dd x = \int_{B^d} f(\sigma_j x) \, W_{\alpha,\gamma}(\sigma_j(x)) \dd x$.
As $W_{\alpha,\gamma}$ is $\sigma_j$-invariant, part \ref{it:odd-has-zero-mean} follows.

Part \ref{it:V-flip-invariance} is proven similarly, using additionally the fact that the composition with $\sigma_j$ preserves the degree of a polynomial.
\end{proof}
\end{proposition}

Given any $\alpha \in \Real$ and $\gamma \in \Real^d$ we introduce the differential-difference operators $\mathcal{D}\psp{\alpha,\gamma;\star}_j$, $j \in \{1, \dotsc, d\}$, by
\begin{multline}\label{parameter-lowering-operator}
\mathcal{D}\psp{\alpha,\gamma;\star}_j q(x) := -(1-\norm{x}^2)^{-\alpha} \, \mathcal{D}\psp{\gamma}_j \! \left( (1-\norm{x}^2)^{\alpha+1} q(x) \right)\\
= -(1-\norm{x}^2) \, \mathcal{D}\psp{\gamma}_j q(x) + 2(\alpha+1)x_j q(x).
\end{multline}
From the inclusions in \eqref{Dunkl-inclusions} they inherit
\begin{equation}\label{Dunkl-star-inclusions}
\mathcal{D}\psp{\alpha,\gamma;\star}_j \left( \CC^m(\overline{B^d}) \right) \subseteq \CC^{m-1}(\overline{B^d})
\quad\text{and}\quad
\mathcal{D}\psp{\alpha,\gamma;\star}_j \left( \poly^d_m \right) \subseteq \poly^d_{m+1}
\end{equation}
for $m \in \natural$ and $m \in \natural_0$, respectively.
Also, from \eqref{Dunkl-sym-skew} and \eqref{xi-sym-skew},
\begin{multline}\label{Dunkl-star-sym-skew}
\mathcal{D}\psp{\alpha,\gamma;\star}_j \Sym_i = \begin{cases} \Sym_i \mathcal{D}\psp{\alpha,\gamma;\star}_j & \text{if } i \neq j, \\ \Skew_i \mathcal{D}\psp{\alpha,\gamma;\star}_j & \text{if } i = j \end{cases}
\quad\text{and}\\
\mathcal{D}\psp{\alpha,\gamma;\star}_j \Skew_i = \begin{cases} \Skew_i \mathcal{D}\psp{\alpha,\gamma;\star}_j & \text{if } i \neq j, \\ \Sym_i \mathcal{D}\psp{\alpha,\gamma;\star}_j & \text{if } i = j. \end{cases}
\end{multline}
As its notation suggests, the $\mathcal{D}\psp{\alpha,\gamma;\star}_j$ operator is indeed adjoint to the Dunkl operator $\mathcal{D}\psp{\gamma}_j$, to the extent allowed by the first part of the following proposition, analogous to \cite[Prop.~3.2]{Figueroa:2017a}, that also goes on to show that $\mathcal{D}\psp{\alpha,\gamma;\star}_j$ is a parameter-lowering and degree-raising operator, that $\mathcal{D}\psp{\gamma}_j$ is a parameter-raising and degree-lowering operator and a useful commutation relation between projections onto orthogonal polynomials spaces and a Dunkl operator.

\begin{proposition}\label{pro:diff-shift}
Let $\alpha \in (-1,\infty)$, $\gamma \in (-1,\infty)^d$ and $j \in \{1,\dotsc,d\}$.
\begin{enumerate}
\item\label{it:adjoint-relation} Let $p, q \in \CC^1(\overline{B^d})$.
Then, $\langle \mathcal{D}\psp{\gamma}_j p,q\rangle_{\alpha+1,\gamma}=\langle p,\mathcal{D}\psp{\alpha,\gamma;\star}_j q \rangle_{\alpha,\gamma}$.
\item\label{it:parameter-lowering} Let $r_k \in \mathcal{V}_{k}^{\alpha+1,\gamma}$.
Then, $\mathcal{D}\psp{\alpha,\gamma;\star}_j r_k \in \mathcal{V}_{k+1}^{\alpha,\gamma}$.
\item\label{it:parameter-raising} Let $p_k \in \mathcal V_k^{\alpha,\gamma}$. Then, $\mathcal{D}\psp{\gamma}_j p_k \in \mathcal V_{k-1}^{\alpha+1,\gamma}$.
\item\label{diff-shift} Let $u \in \CC^1(\overline{B^d})$.
Then, $\mathcal{D}\psp{\gamma}_j\proj\sp{\alpha,\gamma}_k(u) = \proj\sp{\alpha+1,\gamma}_{k-1}(\mathcal{D}\psp{\gamma}_j u)$.
\end{enumerate}
\begin{proof}
As both $\mathcal{D}\psp{\gamma}_j$ and $\mathcal{D}\psp{\alpha,\gamma;\star}_j$ flip $\sigma_j$-symmetry into $\sigma_j$-antisymmetry and vice versa (cf.\ \eqref{Dunkl-sym-skew} and \eqref{Dunkl-star-sym-skew}), per part \ref{it:odd-has-zero-mean} of \autoref{pro:flip}, it is enough to prove \ref{it:adjoint-relation} in the special cases where $p$ and $q$ are either $\sigma_j$-even and $\sigma_j$-odd or $\sigma_j$-odd and $\sigma_j$-even, respectively.
Let us define, for $\delta, \varepsilon > 0$, the set $X_{\delta,\varepsilon} := \{ x \in B^d \mid \abs{x_j} > \delta \ \wedge\ (\forall\,i \in \{1,\dotsc,d\}) \setminus \{j\} \ \abs{x_i} > \varepsilon \}$.
By integration by parts,
\begin{multline}\label{proper-IBP-adj}
\int_{X_{\delta,\varepsilon}} \partial_j p(x) q(x) W_{\alpha+1,\gamma}(x) \dd x\\
= \underbrace{\int_{\partial X_{\delta,\varepsilon}} p(x) q(x) W_{\alpha+1,\gamma}(x) \nu_j(x) \dd S(x)}_{:= b_{\delta,\varepsilon}} - \int_{X_{\delta,\varepsilon}} p(x) \partial_j(q(x) W_{\alpha+1,\gamma}(x)) \dd x,
\end{multline}
where $\nu$ is the outher normal vector field defined almost anywhere (with respect to the surface measure) on $\partial X_{\delta,\varepsilon}$.
Now, for every $x \in X_{\delta,\varepsilon}$, by direct computation
\begin{multline}\label{expansion-adj}
\partial_j(q(x) W_{\alpha+1,\gamma}(x))\\
= \left(\partial_j q(x) (1-\norm{x}^2) - 2(\alpha+1) x_j q(x)\right) W_{\alpha,\gamma}(x) + \frac{\gamma_j}{x_j} q(x) W_{\alpha+1,\gamma}(x).
\end{multline}
From the definition \eqref{parameter-raising-operator} of $\mathcal{D}\psp{\gamma}_j$, \eqref{proper-IBP-adj} and \eqref{expansion-adj},
\begin{multline}\label{IBP-adj}
\int_{X_{\delta,\varepsilon}} \mathcal{D}\psp{\gamma}_j p(x) \, q(x) \, W_{\alpha+1,\gamma}(x) \dd x\\
= b_{\delta,\varepsilon}
- \int_{X_{\delta,\varepsilon}} p(x) \left( \partial_j q(x) \, (1-\norm{x}^2) - 2(\alpha+1) x_j \, q(x) \right) W_{\alpha,\gamma}(x) \dd x\\
-\frac{\gamma_j}{2} \int_{X_{\delta,\varepsilon}} \frac{p(x)+p(\sigma_j x)}{x_j} q(x) \, W_{\alpha+1,\gamma}(x) \dd x.
\end{multline}
As $X_{\delta,\varepsilon}$ and $W_{\alpha,\gamma}$ are $\sigma_j$-invariant, a simple computation shows that
\begin{equation*}
\int_{X_{\delta,\varepsilon}} \frac{p(x) + p(\sigma_j x)}{x_j} q(x) \, W_{\alpha+1,\gamma}(x) \dd x = \int_{X_{\delta,\varepsilon}} p(x) \frac{q(x) - q(\sigma_j x)}{x_j} \, W_{\alpha+1,\gamma}(x) \dd x,
\end{equation*}
which, substituted into \eqref{IBP-adj}, results in (cf.\ \eqref{parameter-lowering-operator})
\begin{equation}\label{almost-adj}
\int_{X_{\delta,\varepsilon}} \mathcal{D}\psp{\gamma}_j p(x) \, q(x) \, W_{\alpha+1,\gamma}(x) \dd x
= b_{\delta,\varepsilon} + \int_{X_{\delta,\varepsilon}} p(x) \, \mathcal{D}\psp{\alpha,\gamma;\star}_j q(x) \, W_{\alpha,\gamma}(x) \dd x.
\end{equation}
As $W_{\alpha+1,\gamma}$ vanishes on $\partial X_{\delta,\varepsilon} \cap \sphere^{d-1}$ and $\nu_j$ vanishes almost everywhere on each of the sets $\{ x \in \partial X_{\delta,\varepsilon} \mid \abs{x_i} = \varepsilon \}$ for $i \in \{1, \dotsc, d\} \setminus \{j\}$, the boundary integral in \eqref{proper-IBP-adj}, \eqref{IBP-adj} and \eqref{almost-adj} can be written as
\begin{multline*}
b_{\delta,\varepsilon} = \int_{\left\{x \in \partial X_{\delta,\varepsilon} \mid \abs{x_j} = \delta\right\}} p(x) \, q(x) \,W_{\alpha+1,\gamma}(x) \operatorname{sign}(x_j) \dd S(x)\\
= \int_{\left\{x \in \partial X_{\delta,\varepsilon} \mid \abs{x_j} = \delta\right\}} \frac{p(x) \, q(x)}{x_j} \, W_{\alpha+1,\gamma+e_j}(x) \dd S(x).
\end{multline*}
Since $pq$ is $\sigma_j$-odd, we infer from \autoref{pro:Hadamard} that $x \mapsto p(x) q(x) / x_j = \rho_j(pq)/2$ belongs to $\CC(\overline{B^d})$.
Also, as $\alpha + 1 > 0$, $(1-\norm{x}^2)^{\alpha+1} \leq 1$ for all $x$ in the integration domain above.
Additionally, said integration domain is contained in $\{x \in [-1,1]^d \mid \abs{x_j} = \delta\}$.
Thus,
\begin{equation*}
\abs{b_{\delta,\varepsilon}}
\leq \delta^{\gamma_j+1} \max_{x \in \overline{B^d}} \abs{ \frac{p(x) q(x)}{x_j} } \prod_{\substack{i=1\\i\neq j}}^d \int_{[-1,1]} \abs{x_i}^{\gamma_i} \dd x_i.
\end{equation*}
Then, as $\gamma_i > -1$ for $i \in \{1, \dotsc, d\}$, for every fixed $\varepsilon$, $\lim_{\delta \to 0^+} b_{\varepsilon,\delta} = 0$.
Then, \ref{it:adjoint-relation} follows from \eqref{almost-adj} by first taking the limit as $\delta \to 0^+$ (which makes the boundary integral disappear) and then the limit as $\varepsilon \to 0^+$ (the volume integrals over $X_{\delta,\epsilon}$ converging to the corresponding ones over $B^d$ by the dominated convergence theorem)

Given $r_k \in \mathcal V_k^{\alpha+1,\gamma}$, by \eqref{Dunkl-star-inclusions}, $\mathcal{D}\psp{\alpha,\gamma;\star}_j r_k \in \poly_{k+1}^d $, and, on account of part \ref{it:adjoint-relation}, the latter is $\LL_{\alpha,\gamma}^2$-orthogonal to $\poly_k^d$, whence part \ref{it:parameter-lowering}.
An analogous argument accounts for part \ref{it:parameter-raising}.

Given $u \in \CC^1(\overline{B^d})$, by part \ref{it:parameter-raising}, $\mathcal{D}\psp{\gamma}_j \proj\sp{\alpha,\gamma}_k(u) \in \mathcal{V}\sp{\alpha+1,\gamma}_{k-1}$.
Part \ref{diff-shift} then comes about from the fact that for all $r \in \mathcal{V}\sp{\alpha+1,\gamma}_{k-1}$,
\begin{multline*}
\langle \mathcal{D}\psp{\gamma}_j\proj\sp{\alpha,\gamma}_k(u), r \rangle_{\alpha+1,\gamma}
\stackrel{\text{\ref{it:adjoint-relation}}}{=} \langle \proj\sp{\alpha,\gamma}_k(u), \mathcal{D}\psp{\alpha,\gamma;\star}_j r \rangle_{\alpha,\gamma}\\
\stackrel{\text{\ref{it:parameter-lowering}}}{=} \langle u, \mathcal{D}\psp{\alpha,\gamma;\star}_j r \rangle_{\alpha,\gamma}
\stackrel{\text{\ref{it:adjoint-relation}}}{=} \langle \mathcal{D}\psp{\gamma}_j u, r \rangle_{\alpha+1,\gamma}.
\end{multline*}
\end{proof}
\end{proposition}

Given $\gamma \in \Real^d$ we introduce the differential-difference operators $\mathcal{D}\psp{\gamma}_{i,j}$, $i, j \in \{1, \dotsc, d\}$, by
\begin{equation}\label{angular-Dunkl}
\mathcal{D}\psp{\gamma}_{i,j} := x_i \, \mathcal{D}\psp{\gamma}_j - x_j \, \mathcal{D}\psp{\gamma}_i.
\end{equation}
Under this definition, the $\mathcal{D}\psp{\gamma}_{i,i}$ operators are simply the null operator.
If $\gamma = 0$ and $i<j$, the $\mathcal{D}\psp{\gamma}_{i,j}$ operators are angular derivatives \cite[Sec.~1.8]{DaiXu:2013}.

The following proposition shows that this operators is minus its adjoint in a certain sense, that this operator is parameter- and degree-invariant and a commutation relation involving this operator and projectors onto the same orthogonal polynomial spaces.

\begin{proposition}\label{pro:angular-Dunkl-adjoint}
Let $\alpha \in (-1,\infty)$, $\gamma \in (-1,\infty)^d$, $i,j \in \{1,\dotsc,d\}$.
\begin{enumerate}
\item\label{it:angular-adjoint} Let $p,q \in \CC^1(\overline{B^d})$.
Then, $\langle \mathcal{D}\psp{\gamma}_{i,j} p , q \rangle_{\alpha,\gamma} = - \langle p , \mathcal{D}\psp{\gamma}_{i,j} q \rangle_{\alpha,\gamma}$.
\item\label{it:angular-parameter-preserving} Let $p_k \in \mathcal{V}\sp{\alpha,\gamma}_k$.
Then, $\mathcal{D}\psp{\gamma}_{i,j} p_k \in \mathcal{V}\sp{\alpha,\gamma}_k$.
\item\label{it:angular-proj-commute} Let $u \in \CC^1(\overline{B^d})$.
Then, $\mathcal{D}\psp{\gamma}_{i,j} \proj\sp{\alpha,\gamma}_k(u) = \proj\sp{\alpha,\gamma}_k(\mathcal{D}\psp{\gamma}_{i,j} u)$.
\end{enumerate}
\begin{proof}
In the non-trivial case $i \neq j$, we infer from the commutation relations \eqref{Dunkl-sym-skew} and \eqref{xi-sym-skew} and the definition \eqref{angular-Dunkl} that the operator $\mathcal{D}\psp{\gamma}_{i,j}$ flips both the $\sigma_i$-parity and the $\sigma_j$ parity of each term in the four-way decomposition \eqref{twofold-parity-decomposition} of $p$.
Then, by part \ref{it:odd-has-zero-mean} of \autoref{pro:flip},
\begin{multline*}
\langle \mathcal{D}\psp{\gamma}_{i,j} p , q \rangle_{\alpha,\gamma}
= \langle \mathcal{D}\psp{\gamma}_{i,j}(\Sym_i \Sym_j p), \Skew_i \Skew_j q \rangle_{\alpha,\gamma}\\
+ \langle \mathcal{D}\psp{\gamma}_{i,j}(\Sym_i \Skew_j p), \Skew_i \Sym_j q \rangle_{\alpha,\gamma}\\
+ \langle \mathcal{D}\psp{\gamma}_{i,j}(\Skew_i \Sym_j p), \Sym_i \Skew_j q \rangle_{\alpha,\gamma}\\
+ \langle \mathcal{D}\psp{\gamma}_{i,j}(\Skew_i \Skew_j p), \Sym_i \Sym_j q \rangle_{\alpha,\gamma}.
\end{multline*}
Thus, it is enough to consider the special cases in which $p$ and $q$ are simultaneously of opposite $\sigma_i$- and $\sigma_j$-parity.
Those cases, in turn, are covered by the supposition that $p q$ is simultaneously $\sigma_i$-odd and $\sigma_j$-odd, which we adopt from now on.

By direct computation it is rapidly checked that,
\begin{multline}\label{angular-expansion}
\langle \mathcal{D}\psp{\gamma}_{i,j} p, q \rangle_{\alpha,\gamma} + \langle p, \mathcal{D}\psp{\gamma}_{i,j} q \rangle_{\alpha,\gamma}
= \langle \mathcal{D}\psp{0}_{i,j} p, q \rangle_{\alpha,\gamma} + \langle p, \mathcal{D}\psp{0}_{i,j} q \rangle_{\alpha,\gamma}\\
+ \int_{B^d} \left( \frac{\gamma_j}{2} x_i \frac{p(x) - p(\sigma_j x)}{x_j} - \frac{\gamma_i}{2} x_j \frac{p(x) - p(\sigma_i x)}{x_i} \right) q(x) W_{\alpha,\gamma}(x) \dd x\\
+ \int_{B^d} p(x) \left( \frac{\gamma_j}{2} x_i \frac{q(x) - q(\sigma_j x)}{x_j} - \frac{\gamma_i}{2} x_j \frac{q(x) - q(\sigma_i x)}{x_i} \right) W_{\alpha,\gamma}(x) \dd x.
\end{multline}
As the purely differential operator $\mathcal{D}\psp{0}_{i,j} = x_i \partial_j - x_j \partial_i$ satisfies the relation $\mathcal{D}\psp{0}_{i,j}(pq) = \mathcal{D}\psp{0}_{i,j}(p) \, q + p \, \mathcal{D}\psp{0}_{i,j}(q)$ and vanishes on radial functions,
\begin{multline}
\langle \mathcal{D}\psp{0}_{i,j} p, q \rangle_{\alpha,\gamma} + \langle p, \mathcal{D}\psp{0}_{i,j}(q) \rangle_{\alpha,\gamma}\\
= \int_{B^d} \div \left( p(x) q(x) (1-\norm{x}^2)^\alpha (x_i e_j - x_j e_i) \right) \prod_{k=1}^d \abs{x_k}^{\gamma_k} \dd x.
\end{multline}
Let us define, for $\varepsilon > 0$ and $0 < r < 1$, the set $X_{r,\varepsilon} := \{ x \in r B^d \mid (\forall \, k \in \{1,\dotsc,d\}) \ \abs{x_k} > \varepsilon \}$.
By the Lebesgue dominated convergence theorem and integration by parts,
\begin{multline}\label{div-IBP}
\langle \mathcal{D}\psp{0}_{i,j} p, q \rangle_{\alpha,\gamma} + \langle p, \mathcal{D}\psp{0}_{i,j}(q) \rangle_{\alpha,\gamma}\\
= \lim_{\substack{r\to 1^-\\ \varepsilon\to 0^+}} \Bigg( \int_{\partial X_{r,\varepsilon}} p(x)q(x) W_{\alpha,\gamma}(x) (x_i e_j - x_j e_i) \cdot \nu(x) \dd S(x)\\
- \underbrace{\int_{X_{r,\varepsilon}} p(x) q(x) W_{\alpha,\gamma}(x) (x_i e_j - x_j e_i) \cdot \sum_{l=1}^d \left( x_l^{-1} \gamma_l \, e_l \right) \dd x}_{:= v_{r,\varepsilon}} \Bigg),
\end{multline}
where $\nu$ is the outer unit normal vector field defined almost anywhere (with respect to the surface measure we have denoted by $S$) on $\partial X_{r,\varepsilon}$.
For $k \in \{1, \dotsc, d\}$, let us define the subsurfaces $A_{r,\varepsilon,k} := \{ x \in \partial X_{r,\varepsilon} \mid \abs{x_k} = \varepsilon\}$.
Then, the union $(r \sphere^{d-1} \cap \partial X_{r,\varepsilon}) \cup \bigcup_{k=1}^d A_{r,\varepsilon,k}$ is a decomposition of $\partial X_{r,\varepsilon}$ in sets whose pairwise intersections have zero $S$-measure.
Now, for $S$-almost every $x \in r \sphere^{d-1} \cap \partial X_{r,\varepsilon}$, $\nu(x) = r^{-1} x$, which is orthogonal to $x_i e_j - x_j e_i$, and for $k \in \{1, \dotsc, d\}$, for $S$-almost every $x \in A_{r,\varepsilon,k}$, $\nu(x) = -\operatorname{sign}(x_k) \, e_k$, which is again orthogonal to $x_i e_j - x_j e_i$ if $k \notin \{i,j\}$.
Hence, on defining
\begin{equation*}
\begin{aligned}
I_{r,\varepsilon,j} & := -\int_{A_{r,\varepsilon,j}} p(x) q(x) x_i \operatorname{sign}(x_j) W_{\alpha,\gamma}(x) \dd S(x),\\
I_{r,\varepsilon,i} & := \int_{A_{r,\varepsilon,i}} p(x) q(x) x_j \operatorname{sign}(x_i) W_{\alpha,\gamma}(x) \dd S(x),
\end{aligned}
\end{equation*}
we can express \eqref{div-IBP} as
\begin{equation}\label{div-IBP-2}
\langle \mathcal{D}\psp{0}_{i,j} p, q \rangle_{\alpha,\gamma} + \langle p, \mathcal{D}\psp{0}_{i,j} q \rangle_{\alpha,\gamma}
= \lim_{\substack{r\to 1^-\\ \varepsilon\to 0^+}} \left( I_{r,\varepsilon,j} + I_{r,\varepsilon,i} - v_{r,\varepsilon} \right).
\end{equation}

As $pq$ is $\sigma_j$-odd, by \autoref{pro:Hadamard}, $x \mapsto p(x) q(x) / x_j = \rho_j(pq)/2$ belongs to $\CC(\overline{B^d})$.
Also, for all $x \in A_{r,\varepsilon,j}$, $\norm{x} \leq r < 1$, which in turn implies that $(1-\norm{x}^2)^\alpha$ is bounded by $(1-r^2)^\alpha$ if $\alpha < 0$ and by $1$ if $\alpha \geq 0$.
Further, $A_{r,\varepsilon,j}$ is contained in $\{ x \in [-1,1]^d \mid \abs{x_j} = \varepsilon \}$.
Thus,
\begin{equation*}
\abs{I_{r,\varepsilon,j}} \leq \varepsilon^{\gamma_j+1} \sup_{x \in \overline{B^d}} \abs{\frac{p(x) q(x)}{x_j}} \, r \, \begin{cases} (1-r^2)^\alpha & \text{if } \alpha < 0\\ 1 & \text{if } \alpha \geq 0 \end{cases} \times \prod_{\substack{k=1\\k\neq j}}^d \int_{[-1,1]} \abs{x_k}^{\gamma_k} \dd x_k.
\end{equation*}
As all the entries of $\gamma$ are greater than $-1$, the integrals over $[-1,1]$ above are finite, so we can conclude that, for all $r \in (0,1)$, $\lim_{\varepsilon \to 0^+} I_{r,\varepsilon,j} = 0$.
The same argument holds for $I_{r,\varepsilon,i}$, so for all $r \in (0,1)$, $\lim_{\varepsilon \to 0^+} I_{r,\varepsilon,i} = 0$.

By expanding the dot product in the integral in $v_{r,\varepsilon}$ (cf.\ \eqref{div-IBP}), judiciously expanding, say, $p = \Sym_i(p) + \Skew_i(p)$ or $p = \Sym_j(p) + \Skew_j(p)$ and changing variable through $\sigma_i$ or $\sigma_j$ where necessary to make $\Sym_i(p)$ and $\Sym_j(p)$ disappear and $\Skew_i(q)$ and $\Skew_j(q)$ appear, we find that
\begin{multline}\label{ang-expand-IBP-volume}
v_{r,\varepsilon} = \int_{X_{r,\varepsilon}} \left( \frac{\gamma_j}{2} x_i \frac{p(x) - p(\sigma_j x)}{x_j} - \frac{\gamma_i}{2} x_j \frac{p(x) - p(\sigma_i x)}{x_i} \right) q(x) W_{\alpha,\gamma}(x) \dd x\\
+ \int_{X_{r,\varepsilon}} p(x) \left( \frac{\gamma_j}{2} x_i \frac{q(x) - q(\sigma_j x)}{x_j} - \frac{\gamma_i}{2} x_j \frac{q(x) - q(\sigma_i x)}{x_i} \right) W_{\alpha,\gamma}(x) \dd x.
\end{multline}
Therefore, substituting \eqref{ang-expand-IBP-volume} into \eqref{div-IBP-2} and the result, in turn, into \eqref{angular-expansion}, yields \ref{it:angular-adjoint}.

Let $p_k \in \mathcal{V}\sp{\alpha,\gamma}_k$.
By \eqref{Dunkl-inclusions}, $\mathcal{D}\psp{\gamma}_{i,j} p_k \in \poly_k^d $, and, on account of part \ref{it:angular-adjoint}, the latter is $\LL_{\alpha,\gamma}^2$-orthogonal to $\poly_{k-1}^d$, whence part \ref{it:angular-parameter-preserving}.

Given $u \in \CC^1(\overline{B^d})$, by part \ref{it:angular-parameter-preserving}, $\mathcal{D}\psp{\gamma}_{i,j} \proj\sp{\alpha,\gamma}_k(u) \in \mathcal{V}\sp{\alpha,\gamma}_k$.
Part \ref{it:angular-proj-commute} then follows from the fact that for all $r \in \mathcal{V}\sp{\alpha,\gamma}_k$,
\begin{multline*}
\langle \mathcal{D}\psp{\gamma}_{i,j} \proj\sp{\alpha,\gamma}_k(u), r \rangle_{\alpha,\gamma}
\stackrel{\text{\ref{it:angular-adjoint}}}{=} - \langle \proj\sp{\alpha,\gamma}_k(u), \mathcal{D}\psp{\gamma}_{i,j} r \rangle_{\alpha,\gamma}\\
\stackrel{\text{\ref{it:angular-parameter-preserving}}}{=} - \langle u, \mathcal{D}\psp{\gamma}_{i,j} r \rangle_{\alpha,\gamma}
\stackrel{\text{\ref{it:angular-adjoint}}}{=} \langle \mathcal{D}\psp{\gamma}_{i,j} u, r \rangle_{\alpha,\gamma}.
\end{multline*}
\end{proof}
\end{proposition}

\section{Sturm--Liouville problems and approximation results}\label{sec:SL-approx}

In rough terms, we will infer from the regularity of a function being approximated the weighted summability of the squared norms of its projectors onto a sequence of orthogonal polynomial spaces.
In turn, this will lead to information about the approximation quality of the truncation projection $S\psp{\alpha,\gamma}_N$.
In this endeavor, the characterization of orthogonal polynomial spaces as eigenspaces of a Sturm--Liouville-type operator will be essential.

From \cite[Th.~8.1.3]{DunklXu:2014}, if $\alpha > -1$ and $\gamma \in (-1,\infty)^d$, every $p_n \in \mathcal{V}^{\alpha,\gamma}_n$ satisfies
\begin{equation}\label{eigenvalue-problem-strong}
\mathcal{L}\sp{\alpha,\gamma}(p_n) := \left(-\Delta_h + (x \cdot \nabla)^2 + 2 \lambda^{\alpha,\gamma} x \cdot \nabla \right) p_n
= n(n+2\lambda^{\alpha,\gamma}) p_n,
\end{equation}
where
\begin{equation}\label{h-Laplacian-lambda}
\Delta_h = \sum_{i=1}^d ( \mathcal{D}\psp{\gamma}_i )^2
\qquad\text{and}\qquad
\lambda^{\alpha,\gamma} = \alpha + \frac{1}{2} \sum_{i=1}^d \gamma_i + \frac{d}{2}.
\end{equation}

We will now put the operator $\mathcal{L}\sp{\alpha,\gamma}$ of \eqref{eigenvalue-problem-strong} into a form that we can test, treat with integration-by-parts substitutes (part \ref{it:adjoint-relation} of \autoref{pro:diff-shift} and part \ref{it:angular-adjoint} of \autoref{pro:angular-Dunkl-adjoint}) and turn into a transparently self-adjoint weak form.

Taking into account the second characterization in \eqref{parameter-lowering-operator} defining $\mathcal{D}\psp{\alpha,\gamma;\star}_j$, it is readily checked that
\begin{equation}\label{sum-Dunkl-adjoint-Dunkl}
	\sum_{i=1}^d \mathcal{D}\psp{\alpha,\gamma;\star}_i (\mathcal{D}\psp{\gamma}_i p)
	= -(1-\norm{x}^2) \Delta_h p + 2(\alpha+1) x \cdot \nabla p + 2(\alpha+1) \sum_{i=1}^d \gamma_i \Skew_i(p).
\end{equation}
Also, from the definition \eqref{angular-Dunkl} and \eqref{Dunkl-xi-commutator}, for all $i, j \in \{1, \dotsc, d\}$ with $i \neq j$,
\begin{multline}\label{gang-squared}
(\mathcal{D}\psp{\gamma}_{i,j})^2
= (x_i^2 (\mathcal{D}\psp{\gamma}_j)^2 + x_j^2 (\mathcal{D}\psp{\gamma}_i)^2) - 2 x_i x_j \, \mathcal{D}\psp{\gamma}_i \mathcal{D}\psp{\gamma}_j\\
- (x_i \, \mathcal{D}\psp{\gamma}_i + x_j \, \mathcal{D}\psp{\gamma}_j) - (\gamma_i x_j \sigma^*_i \mathcal{D}\psp{\gamma}_j + \gamma_j x_i \sigma^*_j \mathcal{D}\psp{\gamma}_i).
\end{multline}
Then, as a direct consequence of \eqref{gang-squared}, we can write
\begin{multline}\label{dunkl-laplace-beltrami-expanding}
\sum_{1 \leq i < j \leq d} (\mathcal{D}\psp{\gamma}_{i,j})^2
= \frac{1}{2} \sum_{\substack{1 \leq i, j \leq d\\i \neq j}} (\mathcal{D}\psp{\gamma}_{i,j})^2\\
= \norm{x}^2 \Delta_h - \sum_{1 \leq i, j \leq d} x_i x_j \, \mathcal{D}\psp{\gamma}_i \mathcal{D}\psp{\gamma}_j - (d-1) \sum_{1 \leq i \leq d} x_i \, \mathcal{D}\psp{\gamma}_i - \sum_{\substack{1 \leq i, j \leq d\\i \neq j}} \gamma_i x_j \sigma^*_i \mathcal{D}\psp{\gamma}_j.
\end{multline}
Considering the easily verifiable identities
\begin{equation}
(x \cdot \nabla)^2 = \sum_{1 \leq i,j \leq d} x_i x_j \partial_i \partial_j + (x \cdot \nabla),
\end{equation}
\begin{equation}
x_i^2 (\mathcal{D}\psp{\gamma}_i)^2 = x_i^2 \partial_i^2 + \gamma_i x_i \partial_i - \gamma_i \Skew_i
\end{equation}
and
\begin{equation}
(x_i \, \mathcal{D}\psp{\gamma}_i)(x_j \, \mathcal{D}\psp{\gamma}_j) = x_i x_j \partial_i \partial _j + (\gamma_j x_i \partial_i \Skew_j + \gamma_i x_j \partial_j \Skew_i ) + \gamma_i \gamma_j \Skew_i \Skew_j,
\end{equation}
for $i \neq j$; we can readily write
\begin{multline}\label{expanding-xixjTiTj}
\sum_{1 \leq i, j \leq d} x_i x_j \, \mathcal{D}\psp{\gamma}_i \mathcal{D}\psp{\gamma}_j
= \sum_{1 \leq i \leq d} x_i^2 (\mathcal{D}\psp{\gamma}_i)^2 + \sum_{\substack{1 \leq i,j \leq d\\i \neq j}} x_i x_j \mathcal{D}\psp{\gamma}_i \mathcal{D}\psp{\gamma}_j\\
= (x \cdot \nabla)^2 - (x \cdot \nabla) + \sum_{1 \leq i \leq d} \gamma_i x_i \partial_i
+ 2 \sum_{\substack{1 \leq i,j \leq d\\i \neq j}} \gamma_i x_j \partial_j \Skew_i\\
- \sum_{1 \leq i \leq d} \gamma_i \Skew_i + \sum_{\substack{1 \leq i,j \leq d\\ i \neq j}} \gamma_i \gamma_j \Skew_i \Skew_j.
\end{multline}
Then, replacing \eqref{expanding-xixjTiTj} in \eqref{dunkl-laplace-beltrami-expanding} and using the fact that $x_j \, \mathcal{D}\psp{\gamma}_j = x_j \partial_j + \gamma_j \Skew_j$, we get
\begin{multline}\label{dunkl-laplace-beltrami-expanding-2}
\sum_{1 \leq i < j \leq d} (\mathcal{D}\psp{\gamma}_{i,j})^2
= \norm{x}^2 \Delta_h - (x \cdot \nabla)^2 - (d-2)(x \cdot \nabla) - (d-2) \sum_{1 \leq i \leq d} \gamma_i \Skew_i\\
- \sum_{\substack{1 \leq i,j \leq d\\i \neq j}} \gamma_i \sigma^*_i x_j \partial_j - \sum_{\substack{1 \leq i,j \leq d\\i \neq j}} \gamma_i \gamma_j \sigma^*_i \Skew_j - \sum_{1 \leq i \leq d} \gamma_i x_i \partial_i\\
- 2 \sum_{\substack{1 \leq i,j \leq d\\i \neq j}} \gamma_i x_j \partial_j \Skew_i - \sum_{\substack{1 \leq i,j \leq d\\i \neq j}} \gamma_i \gamma_j \Skew_i \Skew_j.
\end{multline}
Lastly, considering the identity $\sum_{1\leq i \leq d} \gamma_i x_i \partial_i = (\sum_{1 \leq i \leq d} \gamma_i ) (x \cdot \nabla) - \sum_{\substack{1\leq i,j \leq d\\i \neq j}} \gamma_i x_j \partial_j$, adding and substracting the term $(\sum_{1 \leq i \leq d} \gamma_i) \sum_{i \leq i \leq d} \gamma_i \Skew_i$, considering the identity $\Skew_j = \operatorname{Id} - \sigma^*_j - \Skew_j$, and simplifying, we can readily obtain
\begin{multline}\label{dunkl-laplace-beltrami-newform}
\sum_{1 \leq i < j \leq d} (\mathcal{D}\psp{\gamma}_{i,j})^2
= \norm{x}^2 \Delta_h - (x \cdot \nabla)^2 - \left(d-2+\sum_{i=1}^d \gamma_i\right) (x \cdot \nabla)\\
- \left(d-2+\sum_{i=1}^d \gamma_i\right) \sum_{i=1}^d \gamma_i \Skew_i + \sum_{i=1}^d \sum_{j=1}^d \gamma_i \gamma_j \Skew_i \Skew_j.
\end{multline}
Thus, substracting \eqref{dunkl-laplace-beltrami-newform} from \eqref{sum-Dunkl-adjoint-Dunkl} to then note the appearance of the operator $\mathcal{L}\sp{\alpha,\gamma}$ of \eqref{eigenvalue-problem-strong} we can conclude that it can also be expressed as
\begin{multline}\label{L-BF-form}
\mathcal{L}\sp{\alpha,\gamma}(p) = \sum_{i=1}^d \mathcal{D}\psp{\alpha,\gamma;\star}_i(\mathcal{D}\psp{\gamma}_i p) - \sum_{1 \leq i < j \leq d} (\mathcal{D}\psp{\gamma}_{i,j})^2 p\\
- 2 \lambda^{\alpha,\gamma} \sum_{i=1}^d \gamma_i \Skew_i(p) + \sum_{i=1}^d \sum_{j=1}^d \gamma_i \gamma_j \Skew_i(\Skew_j(p)).
\end{multline}

Using part \ref{it:odd-has-zero-mean} of \autoref{pro:flip}, part \ref{it:adjoint-relation} of \autoref{pro:diff-shift} and part \ref{it:angular-adjoint} of \autoref{pro:angular-Dunkl-adjoint}, we find that
\begin{equation}\label{L-self-adjoint}
\left(\forall\,p\in\CC^2(\overline{B^d})\right)\ \left(\forall\,q\in\CC^1(\overline{B^d})\right) \quad
\langle \mathcal{L}\sp{\alpha,\gamma}(p), q \rangle_{\alpha,\gamma}
= B(p,q),
\end{equation}
where the symmetric bilinear form $B \colon \CC^1(\overline{B^d}) \times \CC^1(\overline{B^d}) \to \Real$ is defined by
\begin{multline}\label{Sturm-Liouville-bilinear-form}
B(u,v)
:= \sum_{i=1}^d \langle \mathcal{D}\psp{\gamma}_i u, \mathcal{D}\psp{\gamma}_i v \rangle_{\alpha+1,\gamma}
+ \sum_{1 \leq i < j \leq d} \langle \mathcal{D}\psp{\gamma}_{i,j} u, \mathcal{D}\psp{\gamma}_{i,j} v \rangle_{\alpha,\gamma}\\
- 2 \lambda^{\alpha,\gamma} \sum_{i=1}^d \gamma_i \langle \Skew_i(u), \Skew_i(v) \rangle_{\alpha,\gamma}\\
+ \sum_{i=1}^d \sum_{j=1}^d \gamma_i \gamma_j \langle \Skew_i(\Skew_j(u)), \Skew_i(\Skew_j(v)) \rangle_{\alpha,\gamma}.
\end{multline}
Through \eqref{L-self-adjoint} the eigenvalue (Sturm--Liouville) problem \eqref{eigenvalue-problem-strong} satisfied by the $\LL^2_{\alpha,\gamma}$-orthogonal polynomials can be expressed in the weak form
\begin{equation}\label{weak-Sturm-Liouville}
\left( \forall \, p_n \in \mathcal{V}\sp{\alpha,\gamma}_n \right) \
\left( \forall \,q \in \CC^1({\overline{B^d}}) \right) \quad
B(p_n, q) = n(n+2\lambda^{\alpha,\gamma}) \langle p_n, q \rangle_{\alpha,\gamma},
\end{equation}

Directly from the definition \eqref{Sturm-Liouville-bilinear-form} and standard inequalities follows the bound
\begin{equation}\label{brute-B-bound}
\left(\forall\,u,v\in\CC^1(\overline{B^d})\right) \quad
\abs{B(u,v)} \leq C_B \norm{u}_{\alpha,\gamma;1} \norm{v}_{\alpha,\gamma;1}
\end{equation}
for some $C_B = C_B(\alpha,\gamma) > 0$.
Given any polynomial $p \in \poly^d$, it follows from \eqref{weak-Sturm-Liouville} and \eqref{parseval-L2} that
\begin{equation*}
B(p,p) = \sum_{n=0}^{\operatorname{degree}(p)} n(n+2\lambda^{\alpha,\gamma}) \norm{\proj\sp{\alpha,\gamma}_n(p)}_{\alpha,\gamma}^2
\geq \inf_{n \in \natural_0}\left( n(n+2\lambda^{\alpha,\gamma}) \right) \norm{p}_{\alpha,\gamma}^2.
\end{equation*}
From the definition of $\lambda^{\alpha,\gamma}$ in \eqref{h-Laplacian-lambda} and the fact that $\alpha, \gamma_1, \dotsc, \gamma_d > -1$ it follows that the above infimum is $\min(0, 1+2\lambda^{\alpha,\gamma})$.
Also, because of the bound \eqref{brute-B-bound} and the density of polynomials in $\HH^1_{\alpha,\gamma} \supseteq \CC^1(\overline{B^d})$ (cf.\ \autoref{pro:polynomials-dense}), the above inequality can be extended to $\CC^1(\overline{B^d})$ functions.
Thus, choosing any $K > \max(0, -1-2\lambda^{\alpha,\gamma})$, the shifted bilinear form $\tilde B \colon \CC^1(\overline{B^d}) \times \CC^1(\overline{B^d}) \to \Real$, defined by
\begin{equation}\label{shifted-B}
\tilde B(p,q) := B(p,q) + K \langle p , q \rangle_{\alpha,\gamma},
\end{equation}
is an inner product in $\CC^1(\overline{B^d})$;
we denote the induced norm by $\norm{\cdot}_{\tilde B}$.
This allows for defining an \emph{ad hoc} function space in very much the same vein of \autoref{def:DunklSobolevSpace}.

\begin{definition}\label{def:ad-hoc}
We define $\HH_{\tilde B}$ as the topological completion of $(\CC^1(\overline{B^d}), \norm{\cdot}_{\tilde B})$.
\end{definition}

\begin{proposition}\label{pro:continuousEmbedding}
There holds the inclusion $\HH^1_{\alpha,\gamma} \subseteq \HH_{\tilde B}$ and
\begin{equation*}
(\forall\,u\in\HH^1_{\alpha,\gamma}) \quad \norm{u}_{\tilde B} \leq (C_B+K)^{1/2} \norm{u}_{\alpha,\gamma;1};
\end{equation*}
that is, $\HH^1_{\alpha,\gamma}$ is continuously embedded in $\HH_{\tilde B}$.
\begin{proof}
From \autoref{def:DunklSobolevSpace}, every $u \in \HH^1_{\alpha,\gamma}$ is (a class of equivalence of) a Cauchy sequence $(u_n)_{n \in \natural}$ of $\CC^1(\overline{B^d})$ functions with respect to the norm $\norm{\cdot}_{\alpha,\gamma;1}$ of \eqref{DunklSobolevInnerProduct}.
By \eqref{brute-B-bound}, $\norm{u_m - u_n}_{\tilde B} \leq (C_B+K)^{1/2} \norm{u_m - u_n}_{\alpha,\gamma;1} \xrightarrow{m,n \to \infty} 0$, so $u \in \HH_{\tilde B}$ according to \autoref{def:ad-hoc}, and $\norm{u}_{\tilde B} = \lim_{n \to \infty} \norm{u_n}_{\tilde B} \leq (C_B+K)^{1/2} \lim_{n \to \infty} \norm{u_n}_{\alpha,\gamma;1} = (C_B+K)^{1/2} \norm{u}_{\alpha,\gamma;1}$.
\end{proof}
\end{proposition}

In the sequence of results \autoref{lem:Parseval-HB}, \autoref{lem:regularitySummability} and \autoref{cor:L2-truncation-error} below, we will exploit the Sturm--Liouville-type equations satisfied by our orthogonal polynomial spaces, both in its strong ($\mathcal{L}\sp{\alpha,\gamma}$-based) and weak ($B$ and $\tilde B$-based) forms, to prove that Dunkl--Sobolev regularity implies convergence rates of our truncation projector, with the error measured in $\LL^2_{\alpha,\gamma}$.
See \cite[Lem.~2.2, Lem.~2.3 and Cor.~2.4]{Figueroa:2017b} for the corresponding results in the $\gamma = 0$ case.

\begin{lemma}\label{lem:Parseval-HB}
Let $\alpha \in (-1,\infty)$ and $\gamma \in (-1,\infty)^d$.
For all $u \in \HH_{\tilde B}$, the series $\sum_{n=0}^\infty \proj\sp{\alpha,\gamma}_n(u)$ (cf.\ \eqref{spanning-consequence}) converges in $\HH_{\tilde B}$ as well.
There also holds the Parseval identity
\begin{equation*}
\left( \forall \, u \in \HH_{\tilde B} \right) \quad \norm{u}_{\tilde B}^2
= \sum_{n = 0}^\infty \left( n(n+2\lambda^{\alpha,\gamma}) + K \right) \norm{ \proj\sp{\alpha,\gamma}_n (u) }_{\alpha,\gamma}^2.
\end{equation*}
\begin{proof}
By density (cf.\ \autoref{def:ad-hoc}), \eqref{weak-Sturm-Liouville} extends to $q \in \HH_{\tilde B}$.
Adding $K \langle p_n, q \rangle_{\alpha,\gamma}$ to both sides we obtain
\begin{equation*}
(\forall \, p_n \in \mathcal{V}\sp{\alpha,\gamma}_n)\ (\forall \, q \in \HH_{\tilde B}) \quad
\tilde B(p_n,q) = \left( n(n+2\lambda^{\alpha,\gamma}) + K \right) \langle p_n , q \rangle_{\alpha,\gamma}.
\end{equation*}
Polynomials are dense in $\HH_{\tilde B}$.
Indeed, if $s \in \HH_{\tilde B}$ is $\HH_{\tilde B}$-orthogonal to $\poly^d$, by the above equality and the fact that $n(n+2\lambda^{\alpha,\gamma}) + K > 0$ for all $n \in \natural_0$, it follows that $s$ is $\LL_{\alpha,\gamma}^2$-orthogonal to $\poly^d$ as well; i.e, $s = 0$.
Now, as the $\mathcal{V}\psp{\alpha,\gamma}_n$ are finite-dimensional, there exists a Hilbert basis of $\LL^2_{\alpha,\gamma}$ consisting of $\LL^2_{\alpha,\gamma}$-orthonormal polynomials.
Such a basis can be renormalized to obtain a Hilbert basis of the closure of polynomials in $\HH_{\tilde B}$; i.e., $\HH_{\tilde B}$ itself.
The desired results then stem from the basic properties of Hilbert bases; see, e.g., \cite[Corollary~5.10]{Brezis}.
\end{proof}
\end{lemma}

\begin{lemma}\label{lem:regularitySummability}
Let $\alpha \in (-1,\infty)$, $\gamma \in (-1,\infty)^d$ and $l \in \natural_0$.
Then, there exists $C = C(\alpha,\gamma,l) > 0$ such that
\begin{equation*}
\left( \forall \, u \in \HH^l_{\alpha,\gamma} \right) \quad \sum_{n = 0}^\infty \left( n(n+2\lambda^{\alpha,\gamma}) + K \right)^l \norm{ \proj\sp{\alpha,\gamma}_n(u) }_{\alpha,\gamma}^2 \leq C \norm{u}_{\alpha,\gamma;l}^2.
\end{equation*}
\begin{proof}
The $l = 0$ case is simply \eqref{parseval-L2}.
From \autoref{pro:continuousEmbedding} and \autoref{lem:Parseval-HB}, for all $u \in \HH^1_{\alpha,\gamma}$,
\begin{equation}\label{1stParsevalEst}
\sum_{n = 0}^\infty \left( n(n+2\lambda^{\alpha,\gamma}) + K \right) \norm{ \proj\sp{\alpha,\gamma}_n(u) }_{\alpha,\gamma}^2
= \norm{u}_{\tilde B}^2
\leq (C_B + K) \norm{u}_{\alpha,\gamma;1}^2,
\end{equation}
which accounts for the $l = 1$ case.

Particularizing \eqref{L-self-adjoint} to $p \in \CC^2(\overline{B^d})$ and $q \in \poly^d$ and using the symmetry of the bilinear form $B$ and the inner product of $\LL^2_{\alpha,\gamma}$, we find that
\begin{equation}\label{L-self-adjoint-2}
(\forall\,p\in\CC^2(\overline{B^d})) \ (\forall\,q\in\poly^d) \quad
\langle \mathcal{L}\sp{\alpha,\gamma}(p), q \rangle_{\alpha,\gamma}
= \langle p, \mathcal{L}\sp{\alpha,\gamma}(q) \rangle_{\alpha,\gamma}.
\end{equation}
Now, by virtue of the bound \eqref{Dunkl-bounded-by-derivatives} and the definitions \eqref{parameter-lowering-operator} and \eqref{angular-Dunkl}, the operators $\mathcal{D}\psp{\gamma}_j$, $\mathcal{D}\psp{\alpha,\gamma;\star}_j$ and $\mathcal{D}\psp{\gamma}_{i,j}$ are bounded operators between $\CC^m(\overline{B^d})$ and $\CC^{m-1}(\overline{B^d})$, $m \geq 1$.
From \autoref{def:DunklSobolevSpace} they extend to bounded operators between $\HH^m_{\alpha,\gamma}$ and $\HH^{m-1}_{\alpha,\gamma}$.
Using these extended first-order operators in the definition of $\mathcal{L}\sp{\alpha,\gamma}$ in \eqref{L-BF-form}, the resulting extended $\mathcal{L}\sp{\alpha,\gamma}$ and $\mathcal{L}\sp{\alpha,\gamma} + K \, I$ operators are bounded maps between $\HH^m_{\alpha,\gamma}$ to $\HH^{m-2}_{\alpha,\gamma}$, $m \geq 2$.
The $m = 2$ case allows for extending \eqref{L-self-adjoint-2} to
\begin{equation}\label{L-self-adjoint-3}
(\forall\,u\in\HH^2_{\alpha,\gamma}) \ (\forall\,q\in\poly^d) \quad
\langle \mathcal{L}\sp{\alpha,\gamma}(u), q \rangle_{\alpha,\gamma}
= \langle u, \mathcal{L}\sp{\alpha,\gamma}(q) \rangle_{\alpha,\gamma}.
\end{equation}
Then, for all $u \in \HH^2_{\alpha,\gamma}$ and $q \in \mathcal{V}\sp{\alpha,\gamma}_n$,
\begin{multline*}
\langle \proj\sp{\alpha,\gamma}_n(\mathcal{L}\sp{\alpha,\gamma}_n(u)), q \rangle_{\alpha,\gamma}
= \langle \mathcal{L}\sp{\alpha,\gamma}_n(u), q \rangle_{\alpha,\gamma}
\stackrel{\eqref{L-self-adjoint-3}}{=} \langle u, \mathcal{L}\sp{\alpha,\gamma}_n(q) \rangle_{\alpha,\gamma}\\
\stackrel{\eqref{eigenvalue-problem-strong}}{=} n(n+2\lambda^{\alpha,\gamma}) \langle u, q \rangle_{\alpha,\gamma}
= n(n+2\lambda^{\alpha,\gamma}) \langle \proj\sp{\alpha,\gamma}_n(u), q \rangle_{\alpha,\gamma},
\end{multline*}
whence
\begin{equation}\label{L-proj-ew-pump}
(\forall\,u\in\HH^2_{\alpha,\gamma}) \quad
\proj\sp{\alpha,\gamma}_n(\mathcal{L}\sp{\alpha,\gamma}(u)) = n(n+2\lambda^{\alpha,\gamma}) \proj\sp{\alpha,\gamma}_n(u).
\end{equation}
Therefore, if $l \geq 2$ is even, our desired result stems from
\begin{multline*}
\sum_{n=0}^\infty \left( n(n+2\lambda\sp{\alpha,\gamma}) + K \right)^l \norm{\proj\sp{\alpha,\gamma}_n(u)}_{\alpha,\gamma}^2\\
\stackrel{\eqref{L-proj-ew-pump}}{=} \sum_{n=0}^\infty \norm{\proj\sp{\alpha,\gamma}_n\left((\mathcal{L}\sp{\alpha,\gamma}+K\,I)^{l/2}(u)\right)}_{\alpha,\gamma}^2
= \norm{(\mathcal{L}\sp{\alpha,\gamma}+K\,I)^{l/2}(u)}_{\alpha,\gamma}^2\\
\leq \norm{(\mathcal{L}\sp{\alpha,\gamma}+K\,I)^{l/2}}_{\mathcal{L}(\HH^l_{\alpha,\gamma},\LL^2_{\alpha,\gamma})}^2 \norm{u}_{\alpha,\gamma;l}^2.
\end{multline*}
Finally, if $l \geq 3$ is odd,
\begin{multline*}
\sum_{n=0}^\infty \left( n(n+2\lambda\sp{\alpha,\gamma}) + K \right)^l \norm{\proj\sp{\alpha,\gamma}_n(u)}_{\alpha,\gamma}^2\\
\stackrel{\eqref{L-proj-ew-pump}}{=} \sum_{n=0}^\infty \left( n(n+2\lambda\sp{\alpha,\gamma}) + K \right) \norm{\proj\sp{\alpha,\gamma}_n\left((\mathcal{L}\sp{\alpha,\gamma}+K\,I)^{(l-1)/2}(u)\right)}_{\alpha,\gamma}^2\\
\stackrel{\eqref{1stParsevalEst}}{=} (C_B + K) \norm{(\mathcal{L}\sp{\alpha,\gamma}+K\,I)^{(l-1)/2}(u)}_{\alpha,\gamma;1}^2\\
\leq (C_B + K) \norm{(\mathcal{L}\sp{\alpha,\gamma}+K\,I)^{(l-1)/2}}_{\mathcal{L}(\HH^l_{\alpha,\gamma},\HH^1_{\alpha,\gamma})}^2 \norm{u}_{\alpha,\gamma;l}^2.
\end{multline*}
\end{proof}
\end{lemma}

\begin{corollary}\label{cor:L2-truncation-error}
For all $\alpha \in (-1,\infty)$, $d \in \natural$, $\gamma \in (-1,\infty)^d$ and $l \in \natural_0$, there exists $C = C(\alpha,\gamma,l)$ such that
\begin{equation*}
(\forall \, N \in \natural_0) \, (\forall \, u \in \HH^l_{\alpha,\gamma}) \quad \norm{u - S\sp{\alpha,\gamma}_N (u)}_{\alpha,\gamma} \leq C (N + 1)^{-l} \norm{u}_{\alpha,\gamma;l}.
\end{equation*}
\begin{proof}
This is a direct consequence of the Parseval identity \eqref{parseval-L2}, \autoref{lem:regularitySummability} and the fact that $n(n+2\lambda^{\alpha,\gamma})+K$ depends quadratically on $n$.
\end{proof}
\end{corollary}

\autoref{approx-ineq} below allows for quantifying the $\LL^2_{\alpha,\gamma}$ norm of a member of $\mathcal{V}\sp{\alpha+1}_k$ with respect to its $\LL^2_{\alpha+1,\gamma}$ norm, thus containing the seed of the quantification of the price to be paid in our main result \autoref{thm:main} because of the mismatch of the orthogonal projector there and the norm the approximation error is measured with; its third part is a Dunkl variant of the Markov brothers' inequality.
However, we need the following technical proposition first.

\begin{proposition}\label{pro:difference-bound}
Let $\alpha \in (-1,\infty)$ and $\gamma \in (-1,\infty)^d$.
Then, there exists $M_{\alpha,\gamma} > 0 $ such that
\begin{multline*}
(\forall\,p\in\LL^2_{\alpha,\gamma}) \quad
- 2 \lambda^{\alpha,\gamma} \sum_{i=1}^d \gamma_i \norm{\Skew_i(p)}_{\alpha,\gamma}^2\\
+ \sum_{i=1}^d \sum_{j=1}^d \gamma_i \gamma_j \norm{\Skew_i(\Skew_j(p))}_{\alpha,\gamma}^2
\geq - M_{\alpha,\gamma} \norm{p}_{\alpha,\gamma}^2.
\end{multline*}
\begin{proof}
This comes from the fact that the $\Skew_j$ operators are bounded in $\LL^2_{\alpha,\gamma}$.
\end{proof}
\end{proposition}
\begin{proposition}\label{approx-ineq}
Let $\alpha \in (-1,\infty)$ and $\gamma \in (-1,\infty)^d$.
\begin{enumerate}
\item\label{extendedOrthogonality} For all $p,q \in \mathcal V\sp{\alpha+1,\gamma}_k$,
\begin{equation*}
\langle p , q \rangle_{\alpha,\gamma}=\left(\frac{k+d/2+\sum_{j=1}^d \gamma_j/2}{\alpha+1}+1 \right) \langle p , q \rangle_{\alpha+1,\gamma}.
\end{equation*}
\item\label{grad-estimation} Let $k \in \natural_0$.
Then, for all $r \in \mathcal{V}\sp{\alpha,\gamma}_k$,
\begin{equation*}
\norm[n]{\mathcal{D}\psp{\gamma} r}_{\alpha,\gamma} \leq \left( \frac{(k (k+2\lambda^{\alpha,\gamma}) + M_{\alpha,\gamma})(k + \lambda\sp{\alpha,\gamma})}{\alpha+1} \right)^{1/2} \norm{r}_{\alpha,\gamma},
\end{equation*}
where $M_{\alpha,\gamma} > 0$ is that of \autoref{pro:difference-bound}.
If $r$ is, additionally, a radial function, this inequality turns into an equality by replacing $M_{\alpha,\gamma}$ with $0$.
\item\label{markov-ineq} There exists a constant $C = C(\alpha,\gamma) > 0$ such that, for all $n \in \natural_0$ and $p \in \poly^d_n$,
\begin{equation*}
\norm[n]{\mathcal{D}\psp{\gamma} p}_{\alpha,\gamma} \leq C n^2 \norm{p}_{\alpha,\gamma}.
\end{equation*}
\end{enumerate}
\begin{proof}
On homogeneous polynomials of degree $k$, $k \in \natural_0$, there holds $x \cdot \nabla = k\,I$.
As a first consequence, $x \cdot \nabla$ maps $\poly^d_n$ into itself, for every $n \in \natural_0$.

Let $p,q \in \mathcal{V}\sp{\alpha+1,\gamma}_k$.
As every member of $\mathcal{V}\sp{\alpha+1,\gamma}_k$ is a linear combination of homogeneous polynomials of degree ranging from $0$ to $k$, there exists a homogeneous polynomial $s_p$ of degree $k$ such that $p - s_p \in \poly^d_{k-1}$ and hence $x \cdot \nabla p - x \cdot \nabla s_p \in \poly^d_{k-1}$.
Thus,
\begin{equation}\label{nabla_k}
\langle x \cdot \nabla p , q \rangle_{\alpha+1,\gamma} = \langle x \cdot \nabla s_p , q \rangle_{\alpha+1,\gamma} = k \langle s_p , q \rangle_{\alpha+1,\gamma} = k \langle p , q \rangle_{\alpha+1,\gamma}.
\end{equation}
Using the fact that $\div(x) = d$ and \eqref{nabla_k} (which is still valid if the roles of $p$ and $q$ are interchanged),
\begin{multline}\label{div_as_T}
(2k+d) \langle p , q \rangle_{\alpha+1,\gamma}
= \langle x \cdot \nabla p , q \rangle_{\alpha+1,\gamma} + \langle p , x \cdot \nabla q \rangle_{\alpha+1,\gamma} + d \langle p , q \rangle_{\alpha+1,\gamma}\\
= \int_{B^d} \div{\left( p(x)q(x)x \right)} W_{\alpha+1,\gamma}(x) \dd x.
\end{multline}
Now,
\begin{multline*}
\int_{B^d} \div{\left( p(x)q(x)x \right)} \, W_{\alpha+1,\gamma}(x) \dd x + \sum_{j=1}^d \gamma_j \langle p , q \rangle_{\alpha+1,\gamma }
= \sum_{j=1}^d \langle \mathcal{D}\psp{\gamma}_j (x_j pq) , 1 \rangle_{\alpha+1,\gamma}\\
= \sum_{j=1}^d \langle x_j p q , \mathcal{D}\psp{\alpha,\gamma;\star}_j(1) \rangle_{\alpha,\gamma}
= 2(\alpha+1)\int_{B^d} p(x)q(x) \norm{x}^2 W_{\alpha,\gamma}(x) \dd x,
\end{multline*}
where the first equality comes from the definition \eqref{parameter-raising-operator} and part \ref{it:odd-has-zero-mean} of \autoref{pro:flip}, the second from part \ref{it:adjoint-relation} of \autoref{pro:diff-shift} and the third from the definition \eqref{parameter-lowering-operator}.
Substituting this into \eqref{div_as_T}, yields
\begin{equation*}
(2k+d) \langle p , q \rangle_{\alpha+1,\gamma} = 2(\alpha+1)\int_{B^d} p(x)q(x) \norm{x}^2 W_{\alpha,\gamma}(x) \dd x - \sum_{j=1}^d \gamma_j \langle p , q \rangle_{\alpha+1,\gamma}.
\end{equation*}
Part \ref{extendedOrthogonality} then follows from the fact that $W_{\alpha,\gamma}(x)=\norm{x}^2 W_{\alpha,\gamma}(x) + W_{\alpha+1,\gamma}(x)$.

Part \ref{grad-estimation} is obviously true if $k = 0$; otherwise, from part \ref{it:parameter-raising} of \autoref{pro:diff-shift} and part \ref{extendedOrthogonality} above,
\begin{equation}\label{extendedOrthogonality-aux}
(\forall\,r \in \mathcal{V}\sp{\alpha,\gamma}_k) \quad
\norm[n]{\mathcal{D}\psp{\gamma} r}_{\alpha,\gamma}^2
= \frac{k + \lambda\sp{\alpha,\gamma}}{\alpha + 1} \norm[n]{\mathcal{D}\psp{\gamma} r}_{\alpha+1,\gamma}^2.
\end{equation}
On the other hand, from \eqref{Sturm-Liouville-bilinear-form} and \eqref{weak-Sturm-Liouville} (with $p_n$ and $q$ there both set as $r$),
\begin{multline*}
\norm[n]{\mathcal{D}\sp{\gamma} r}_{\alpha+1,\gamma}^2
+ \sum_{1\leq i<j\leq d} \norm[n]{\mathcal{D}\sp{\gamma}_{i,j} r}_{\alpha,\gamma}^2
- 2 \lambda^{\alpha,\gamma} \sum_{i=1}^d \gamma_i \norm{\Skew_i(r)}_{\alpha,\gamma}^2\\
+ \sum_{i=1}^d \sum_{j=1}^d \gamma_i \gamma_j \norm{\Skew_i(\Skew_j(r))}_{\alpha,\gamma}^2
+ M_{\alpha,\gamma} \norm{r}_{\alpha,\gamma}^2\\
= \left(k(k+2\lambda^{\alpha,\gamma}) + M_{\alpha,\gamma}\right) \norm{r}_{\alpha,\gamma}^2.
\end{multline*}
Per \autoref{pro:difference-bound}, dropping the second, third, fourth and fifth terms from the left-hand side of the above equality, the remaining first term will be bounded from above by the right-hand side.
Combining the resulting inequality with \eqref{extendedOrthogonality-aux} and taking square roots results in the generic case of part \ref{grad-estimation}.
If $r$ is radial, the second, third and fourth terms on the left-hand side above vanish, and $M_{\alpha,\gamma}$ can be canceled from both sides; what now remains an equality can also be combined with \eqref{extendedOrthogonality-aux}.

Given $n \in \natural_0$ and $p \in \poly^d_n$, from \eqref{spanning-consequence}, part \ref{grad-estimation} above, and the Cauchy--Schwarz inequality,
\begin{align*}
\norm[n]{\mathcal{D}\psp{\gamma} p}_{\alpha,\gamma}
& \leq \sum_{k=0}^n \norm[n]{\mathcal{D}\psp{\gamma} \proj\sp{\alpha,\gamma}_k (p)}_{\alpha,\gamma}\\
& \leq \left( \sum_{k=0}^n \frac{(k (k+2\lambda\sp{\alpha,\gamma}) + M_{\alpha,\gamma})(k + \lambda\sp{\alpha,\gamma})}{\alpha+1} \right)^{1/2} \left( \sum_{k=0}^n \norm{\proj\sp{\alpha,\gamma}_k (p)}_{\alpha,\gamma}^2 \right)^{1/2}\\
& = \left( \frac{(n+1)(n + 2 \lambda\sp{\alpha,\gamma}) (n^2 + 2 \lambda\sp{\alpha,\gamma}n + n +2 M_{\alpha,\gamma})}{4 (\alpha+1)} \right)^{1/2} \norm{p}_{\alpha,\gamma}.
\end{align*}
Part \ref{markov-ineq} then follows after realizing that there exists a positive constant $C$ depending on $\alpha$ and $\gamma$ only such that $\frac{(n+1)(n + 2 \lambda\sp{\alpha,\gamma}) (n^2 + 2 \lambda\sp{\alpha,\gamma}n + n +2 M_{\alpha,\gamma})}{4 (\alpha+1)} \leq C^2 n^4$ for all $n \in \natural_0$.
\end{proof}
\end{proposition}

Now we prove a lemma with the core of the main result, a bridging corollary and then, finally, the main result itself.

\begin{lemma}\label{lem:T-proj-commute}
Let $\alpha \in (-1,\infty)$, $\gamma \in (-1,\infty)^d$ and $l \in \natural$.
Then, there exists $C = C(\alpha,\gamma,l) > 0$ such that for all $u \in \HH^l_{\alpha,\gamma}$, $n \in \natural$ and $j \in \{1,\dotsc,d\}$,
\begin{equation*}
\norm{ \mathcal{D}\psp{\gamma}_j S\sp{\alpha,\gamma}_n (u) - S\sp{\alpha,\gamma}_n (\mathcal{D}\psp{\gamma}_j u) }_{\alpha,\gamma}
\leq C \, n^{3/2 - l} \norm{ \mathcal{D}\psp{\gamma}_j u }_{\alpha,\gamma;l-1}.
\end{equation*}
\begin{proof}
Let us first assume that $u \in \CC^l (\overline{B^d})$.
Combining part \ref{it:id-shift} of \autoref{pro:id-shift} and \ref{diff-shift} of \autoref{pro:diff-shift}, we obtain
\begin{multline}\label{T-proj-commute}
\mathcal{D}\psp{\gamma}_j \proj\sp{\alpha,\gamma}_{k+1} (u) - \proj\sp{\alpha,\gamma}_k (\mathcal{D}\psp{\gamma}_j u)\\
= \proj\sp{\alpha+1,\gamma}_k \circ \proj\sp{\alpha,\gamma}_{k+2} (\mathcal{D}\psp{\gamma}_j u) - \proj\sp{\alpha+1,\gamma}_{k-2} \circ \proj\sp{\alpha,\gamma}_k (\mathcal{D}\psp{\gamma}_j u)
\end{multline}
Using \eqref{spanning-consequence} to express $S\sp{\alpha,\gamma}_n$ in terms of the $\proj\sp{\alpha,\gamma}_k$, using \eqref{T-proj-commute}, noticing that a telescoping sum results and using part \ref{it:unnamed} of \autoref{pro:id-shift} to expand an appearance of $\proj\sp{\alpha,\gamma}_n (\mathcal{D}\psp{\gamma}_j u) \in \mathcal{V}\sp{\alpha,\gamma}$,
\begin{multline}\label{commute-diff}
\mathcal{D}\psp{\gamma}_j S\sp{\alpha,\gamma}_n (u) - S\sp{\alpha,\gamma}_n (\mathcal{D}\psp{\gamma}_j u)
= \sum_{k = 0}^n \mathcal{D}\psp{\gamma}_j \proj\sp{\alpha,\gamma}_k(u) - \sum_{k = 0}^n \proj\sp{\alpha,\gamma}_k (\mathcal{D}\psp{\gamma}_j u)\\
= \sum_{k = 0}^{n-1} \left( \mathcal{D}\psp{\gamma}_j \proj\sp{\alpha,\gamma}_{k+1} (u) - \proj\sp{\alpha,\gamma}_k (\mathcal{D}\psp{\gamma}_j u) \right) - \proj\sp{\alpha,\gamma}_n (\mathcal{D}\psp{\gamma}_j u)\\
= \proj\sp{\alpha+1,\gamma}_{n-2} \circ \proj\sp{\alpha,\gamma}_n (\mathcal{D}\psp{\gamma}_j u) + \proj\sp{\alpha+1,\gamma}_{n-1} \circ \proj\sp{\alpha,\gamma}_{n+1} (\mathcal{D}\psp{\gamma}_j u) - \proj\sp{\alpha,\gamma}_n (\mathcal{D}\psp{\gamma}_j u)\\
= \proj\sp{\alpha+1,\gamma}_{n-1} \circ \proj\sp{\alpha,\gamma}_{n+1} (\mathcal{D}\psp{\gamma}_j u) - \proj\sp{\alpha+1,\gamma}_n \circ \proj\sp{\alpha,\gamma}_n (\mathcal{D}\psp{\gamma}_j u).
\end{multline}
Now, by part \ref{extendedOrthogonality} of \autoref{approx-ineq}, the fact that $\norm[n]{\proj\sp{\alpha+1,\gamma}_{n-1}}_{\mathcal{L} (\LL^2_{\alpha+1,\gamma})} \leq 1$ and the fact that $\norm{\cdot}_{\alpha+1,\gamma} \leq \norm{\cdot}_{\alpha,\gamma}$ in $\LL^2_{\alpha,\gamma}$ (because $W_{\alpha+1,\gamma} \leq W_{\alpha,\gamma}$) we have that, for all $n \geq 1$,
\begin{multline}\label{proj-bound-1}
\norm[n]{\proj\sp{\alpha+1,\gamma}_{n-1} \circ \proj\sp{\alpha,\gamma}_{n+1} (\mathcal{D}\psp{\gamma}_j u)}_{\alpha,\gamma}^2\\
\leq \frac{n + d/2 + \sum_{j=1}^d \gamma_j / 2 + \alpha}{\alpha + 1} \norm[n]{\proj\sp{\alpha,\gamma}_{n+1} (\mathcal{D}\psp{\gamma}_j u)}_{\alpha,\gamma}^2.
\end{multline}
Analogous arguments show that, for all $n \in \natural$,
\begin{multline}\label{proj-bound-2}
\norm[n]{\proj\sp{\alpha+1,\gamma}_n \circ \proj\sp{\alpha,\gamma}_n (\mathcal{D}\psp{\gamma}_j u)}_{\alpha,\gamma}^2\\
\leq \frac{n + 1 + d/2 + \sum_{j=1}^d \gamma_j / 2 + \alpha}{\alpha + 1} \norm[n]{\proj\sp{\alpha,\gamma}_n (\mathcal{D}\psp{\gamma}_j u)}_{\alpha,\gamma}^2.
\end{multline}
Taking the squared $\LL^2_{\alpha,\gamma}$ norm of both ends of \eqref{commute-diff}, exploiting the $\LL_{\alpha,\gamma}^2$ orthogonality of $\mathcal{V}\sp{\alpha+1,\gamma}_{n-1}$ and $\mathcal{V}\sp{\alpha+1,\gamma}_n$ (a consequence of the parity relation \eqref{parity}) and the bounds \eqref{proj-bound-1} and \eqref{proj-bound-2} we observe that
\begin{multline*}
\norm[n]{\mathcal{D}\psp{\gamma}_j S\sp{\alpha,\gamma}_n (u) - S\sp{\alpha,\gamma}_n (\mathcal{D}\psp{\gamma}_j u)}_{\alpha,\gamma}^2\\
\leq \frac{n + 1 +d/2 + \sum_{j=1}^d \gamma_j / 2 + \alpha}{\alpha + 1} \norm[n]{\mathcal{D}\psp{\gamma}_j u - S\sp{\alpha,\gamma}_{n-1} (\mathcal{D}\psp{\gamma}_j u)}_{\alpha,\gamma}^2.
\end{multline*}
As $\mathcal{D}\psp{\gamma}_j u \in \CC^{l-1} (\overline{B^d})$ (cf.\ \autoref{pro:Hadamard}), we can appeal to \autoref{cor:L2-truncation-error} to obtain the desired result for $u \in \CC^l (\overline{B^d})$ after realizing that there exists a constant $\tilde C$ depending only on $\alpha$, $\gamma$ and $l$ such that $\frac{n + 1 +d/2 + \sum_{j=1}^d \gamma_j / 2 + \alpha}{\alpha + 1} (n^{- (l - 1)})^2 \leq \tilde C \, n^{3 - 2 l}$ for all $n \in \natural$.
The general result then follows via density of $\CC^l (\overline{B^d})$ in $\HH^l_{\alpha,\gamma}$ (\autoref{def:DunklSobolevSpace}).
\end{proof}
\end{lemma}

\begin{corollary}\label{grad-projector-commute}
Let $\alpha \in (-1,\infty)$, $\gamma \in (-1,\infty)^d$ and $r , l \in \natural$ with $r \leq l$.
Then, there exists $C = C (\alpha,\gamma,l,r) > 0$ such that, for all $u \in \HH^l_{\alpha,\gamma}$ and $n \in \natural$,
\begin{equation*}
\norm{(\mathcal{D}\psp{\gamma})^r S\sp{\alpha,\gamma}_n (u) - S\sp{\alpha,\gamma}_n ((\mathcal{D}\psp{\gamma})^r u)}_{\alpha,\gamma}
\leq C \, n^{2 r - 1/2 - l} \norm{u}_{\alpha,\gamma;l}.
\end{equation*}
\begin{proof}
Let us first note that iterating part \ref{markov-ineq} of \autoref{approx-ineq} we find that for all $r \in \natural$ there exists $C > 0$ depending on $\alpha$, $\gamma$, and $r$ such that
\begin{equation}\label{iterated-markov}
(\forall \, n \in \natural_0) \ (\forall \, p \in \poly_n^d) \quad \norm[n]{(\mathcal{D}\psp{\gamma})^r p}_{\alpha,\gamma} \leq C \, n^{2 r} \norm{p}_{\alpha,\gamma}.
\end{equation}
We will now operate by induction on $r$.
Taking the square root of the sum with respect to $j$ of the square of both sides of the inequality in \autoref{lem:T-proj-commute} the case $r = 1$ follows almost immediately.
Let us suppose now that our desired result holds for some $r \in \{1,\dotsc,l\}$ and that $r +1 \leq l$.
Then, for all $j \in \{1,\dotsc,d\}$, by the triangle inequality,
\begin{multline*}
\norm{(\mathcal{D}\psp{\gamma})^r \mathcal{D}\psp{\gamma}_j S\sp{\alpha,\gamma}_n (u) - S\sp{\alpha,\gamma}_n ((\mathcal{D}\psp{\gamma})^r \mathcal{D}\psp{\gamma}_j u)}_{\alpha,\gamma}\\
\leq \norm{(\mathcal{D}\psp{\gamma})^r \mathcal{D}\psp{\gamma}_j S\sp{\alpha,\gamma}_n (u) - (\mathcal{D}\psp{\gamma})^r S\sp{\alpha,\gamma}_n (\mathcal{D}\psp{\gamma}_j u )}_{\alpha,\gamma}\\
+ \norm{(\mathcal{D}\psp{\gamma})^r S\sp{\alpha,\gamma}_n (\mathcal{D}\psp{\gamma}_j u) - S\sp{\alpha,\gamma}_n ((\mathcal{D}\psp{\gamma})^r \mathcal{D}\psp{\gamma}_j u)}_{\alpha,\gamma}.
\end{multline*}
By \eqref{iterated-markov} and \autoref{lem:T-proj-commute}, the first term is bounded by an appropriate constant times $n^{2r} n^{3/2 - l} \norm[n]{\mathcal{D}\psp{\gamma}_j u}_{\alpha,\gamma;l-1}$.
By the induction hypothesis and the fact that $\mathcal{D}\psp{\gamma}_j u \in \HH^{l-1}_{\alpha,\gamma}$, the second term is bounded by an appropriate constant times $n^{2r - 1/2 - (l-1)} \norm[n]{\mathcal{D}\psp{\gamma}_j u}_{\alpha,\gamma;l-1}$.
Then, the desired result in the $r+1$ case follows from summing up with respect to $j$ and standard inequalities connecting vector $1$- and $2$-norms.
\end{proof}
\end{corollary}

\begin{proof}[Proof of \autoref{thm:main}]
For every $k \in \{1,\dotsc,r\}$,
\begin{multline*}
\norm{(\mathcal{D}\psp{\gamma})^k u - (\mathcal{D}\psp{\gamma})^k S\sp{\alpha,\gamma}_N (u)}_{\alpha,\gamma}^2\\
\leq 2 \norm{(\mathcal{D}\psp{\gamma})^k u - S\sp{\alpha,\gamma}_N ((\mathcal{D}\psp{\gamma})^k u)}_{\alpha,\gamma}^2 + 2 \norm{S\sp{\alpha,\gamma}_N ((\mathcal{D}\psp{\gamma})^k u) - (\mathcal{D}\psp{\gamma})^k S\sp{\alpha,\gamma}_N (u)}_{\alpha,\gamma}^2\\
\leq C_1 \, (N + 1)^{-2 (l- k)} \sum_{\abs{\beta} = k} \binom{k}{\beta} \norm[n]{\mathcal{D}\psp{\gamma}_\beta u}_{\alpha,\gamma;l-k}^2 + C_2 \, N^{4k - 1 - 2l} \norm{u}_{\alpha,\gamma;l}^2\\
\leq C_3 \, N^{4r - 1 - 2l} \norm{u}_{\alpha,\gamma;l}^2,
\end{multline*}
where we have used \autoref{cor:L2-truncation-error}, \autoref{grad-projector-commute} and $C_1$ and $C_2$ depend on $\alpha$, $\gamma$, $l$ and $k$ only and $C_3$ depends on $\alpha$, $\gamma$, $l$ and $r$ only.
Thus,
\begin{multline*}
\norm{u - S\sp{\alpha,\gamma}_N (u)}_{\alpha,\gamma;r}^2 \leq \left( C_4 \, (N + 1)^{-2l} + r \, C_3 \, N^{4r - 1 - 2l} \right) \norm{u}_{\alpha,\gamma;l}^2\\
\leq C_5 \, N^{4r - 1 - 2l} \norm{u}_{\alpha,\gamma;l}^2,
\end{multline*}
where we have again used \autoref{cor:L2-truncation-error}, $C_4$ depends on $\alpha$, $\gamma$, and $l$ only and $C_5$ depends on $\alpha$, $\gamma$, $l$ and $r$ only.
\end{proof}

\section{On the sharpness of the main result}\label{sec:sharpness}

We will say that our main result, \autoref{thm:main} is \emph{sharp} if the power on the truncation degree $N$ appearing there cannot be lowered.
We refer to \cite[Sec.~5]{Figueroa:2017a} for an account of sharpness results for previous incarnations of our main result, to which we should add that the one-dimensional, Jacobi-weighted variant of \cite[Th.~2.6]{Xu:2018} comes with its own proof of sharpness (for the cases in which, in our notation, $r = l$).

We will the sharpness of our main result for all dimensions $d \in \natural$, natural singularity parameters $\alpha > -1$ and $\gamma \in (-1,\infty)^d$, but restricted to $l = r = 1$.

We will find it easier to work with an alternative norm, equivalent to that of $\HH^1_{\alpha,\gamma}$, as proved in \autoref{pro:H1-equivalence} (see \cite[Lem.~2.6]{Figueroa:2017b} for the corresponding result in the $\gamma = 0$ case).
However, we first need to show that differentiable functions with vanishing Dunkl gradient are constant in $B^d$.

\begin{proposition}\label{pro:dunkl-kernel}
Let $\gamma > -1$, $L>0$, and $p \in \CC^1(-L,L)$ such that
\begin{equation}\label{dunkl-equals-0}
\mathcal{D}\psp{\gamma}_1 p = 0 \quad \text{in $(-L,L)$}.
\end{equation}
Then, $p$ is constant in $(-L,L)$.
\begin{proof}
As $\frac{\Skew p}{x}$ is a always an even function and so is $0$, directly from the definition \eqref{parameter-raising-operator} of $\mathcal{D}\psp{\gamma}_1$, it follows that $p'$ is an even function.
Therefore, $p$ can be expressed as the sum of a constant and an odd function, which also belongs to $\CC^1(-L,L)$.
Hence, $y := \left.\Skew(p)\right|_{(0,L)}$ satisfies the Cauchy--Euler differential equation
\begin{equation*}
x \, y'(x) + \gamma \, y(x) = 0,
\end{equation*}
whence it has the form
\begin{equation*}
y(x) = C \, x^{-\gamma}.
\end{equation*}
As $y$ extends to a $\CC^1(-1,1)$ function, $C$ has to vanish.
\end{proof}
\end{proposition}

\begin{proposition}\label{pro:DunklGrad-kernel}
Let $\gamma \in (-1,\infty)^d$ and $p \in \CC^1(B^d)$ such that
\begin{equation*}
\mathcal{D}\psp{\gamma} p = 0 \quad \text{in $B^d$}.
\end{equation*}
Then, $p$ is constant in $B^d$.
\begin{proof}
Given two points in $B^d$, they can be connected via a polygonal path consisting exclusively of segments that are parallel to a coordinate axis.
By applying \autoref{pro:dunkl-kernel} in every segment, it transpires that $p$ is constant along this polygonal path and, in particular, the evaluations of $p$ at the original two points coincide.
\end{proof}
\end{proposition}

\begin{proposition}\label{pro:H1-equivalence}
The following is an equivalent inner product for $(\CC^1(\overline{B^d}),\langle \cdot,\cdot \rangle_{\alpha,\gamma;1})$.
\begin{equation}\label{special-IP-m}
\langle u, v \rangle_{\alpha, \gamma; 1, \mathrm{P}} := \langle \mathcal{D}\psp{\gamma} u, \mathcal{D}\psp{\gamma} v \rangle_{\alpha,\gamma} + \langle S^{\alpha,\gamma}_0(u), S^{\alpha,\gamma}_0(v) \rangle_{\alpha,\gamma}.
\end{equation}
Therefore the topological completion of $(\CC^1(\overline{B^d}), \langle \cdot,\cdot \rangle_{\alpha,\gamma;1,\mathrm{P}})$ equals $\HH^1_{\alpha,\gamma}$, with the extension of $\langle \cdot,\cdot \rangle_{\alpha,\gamma;1,\mathrm{P}}$ to $\HH^1_{\alpha,\gamma}$ (cf.~\autoref{def:DunklSobolevSpace}) being an equivalent inner product.
\begin{proof}
$\langle \cdot,\cdot \rangle_{\alpha,\gamma;1,\mathrm{P}}$ being an inner product is a direct consequence of \autoref{pro:DunklGrad-kernel}.
Clearly, $\norm{\cdot}_{\alpha,\gamma;1,\mathrm{P}} \leq \norm{\cdot}_{\alpha,\gamma,1}$.

We will now prove the converse bound.
Let $u \in \CC^1(\overline{B^d})$.
Given $N \in \natural$, by Parseval's identity \eqref{parseval-L2},
\begin{equation}\label{parseval-decomposed}
\norm{u}_{\alpha,\gamma}^2
= \norm{S\sp{\alpha,\gamma}_N u}_{\alpha,\gamma}^2 + \sum_{n=N+1}^\infty \norm{ \proj\sp{\alpha,\gamma}_n(u) }_{\alpha,\gamma}^2.
\end{equation}
As $\poly^d_N$ is finite dimensional, there exists a positive constant $C > 0$, depending only on $N$, $\alpha$ and $\gamma$, such that
\begin{equation*}
\left( \forall \, p \in \poly^d_N \right) \quad \norm{p}_{\alpha,\gamma}^2 \leq C \left( \norm{S\sp{\alpha,\gamma}_0 p}_{\alpha,\gamma}^2 + \norm{\mathcal{D}\psp{\gamma} p}_{\alpha+1,\gamma}^2 \right).
\end{equation*}
In particular, with $p = S\sp{\alpha,\gamma}_N u$ and using part \ref{diff-shift} of \autoref{pro:diff-shift}, we have
\begin{multline}\label{finite-dimensional-bound}
\norm{S\sp{\alpha,\gamma}_N u}_{\alpha,\gamma}^2 \leq C \left( \norm{S\sp{\alpha,\gamma}_0 S\sp{\alpha,\gamma}_N u}_{\alpha,\gamma}^2 + \norm{\mathcal{D}\psp{\gamma} S\sp{\alpha,\gamma}_N u}_{\alpha+1,\gamma}^2 \right)\\
= C \left( \norm{S\sp{\alpha,\gamma}_0 u}_{\alpha,\gamma}^2 + \norm{S\sp{\alpha+1,\gamma}_{N-1} \mathcal{D}\psp{\gamma} u}_{\alpha+1,\gamma}^2 \right)
\leq C \left( \norm{S\sp{\alpha,\gamma}_0 u}_{\alpha,\gamma}^2 + \norm{\mathcal{D}\psp{\gamma} u}_{\alpha+1,\gamma}^2 \right) .
\end{multline}
In turn, as $\proj\sp{\alpha,\gamma}_n(u) \in \mathcal{V}\sp{\alpha,\gamma}_n$, by \eqref{weak-Sturm-Liouville}, \eqref{Sturm-Liouville-bilinear-form}, part \ref{diff-shift} of \autoref{pro:diff-shift}, part \ref{it:angular-proj-commute} of \autoref{pro:angular-Dunkl-adjoint} and taking into account that $\norm{\Skew_i \cdot}_{\alpha,\gamma} \leq \norm{\cdot}_{\alpha,\gamma}$ for all $i \in \{1,\dotsc,d\}$, we obtain
\begin{multline}\label{proj-bound}
n(n+2\lambda^{\alpha,\gamma}) \norm{ \proj\sp{\alpha,\gamma}_n(u) }_{\alpha,\gamma}^2
= B(\proj\sp{\alpha,\gamma}_n(u), \proj\sp{\alpha,\gamma}_n(u))\\
= \norm{\mathcal{D}\psp{\gamma} \proj\sp{\alpha,\gamma}_n(u)}_{\alpha+1,\gamma}^2 + \sum_{1 \leq i < j \leq d} \norm{\mathcal{D}\psp{\gamma}_{i,j} \proj\sp{\alpha,\gamma}_n(u)}_{\alpha,\gamma}^2\\
- 2\lambda^{\alpha,\gamma} \sum_{i=1}^d \gamma_i \norm{\Skew_i \proj\sp{\alpha,\gamma}_n(u)}_{\alpha,\gamma}^2 + \sum_{i,j=1}^d \gamma_i \gamma_j \norm{\Skew_i \Skew_j \proj\sp{\alpha,\gamma}_n(u)}_{\alpha,\gamma}^2\\
\leq \norm{\proj\sp{\alpha+1,\gamma}_{n-1} \mathcal{D}\psp{\gamma}(u)}_{\alpha+1,\gamma}^2 + \sum_{1 \leq i < j \leq d} \norm{\proj\sp{\alpha,\gamma}_n \mathcal{D}\psp{\gamma}_{i,j}(u)}_{\alpha,\gamma}^2 + \tilde C \norm{\proj\sp{\alpha,\gamma}_n(u)}_{\alpha,\gamma}^2,
\end{multline}
where $\tilde C = \tilde C(\alpha,\gamma) := 2\abs{\lambda^{\alpha,\gamma}} \sum_{i=1}^d \abs{\gamma_i} + \sum_{i,j=1}^d \abs{\gamma_i \gamma_j}$.
Let us now fix $N \in \natural$ to any value which ensures that $\tilde C < n(n+2\lambda^{\alpha,\gamma})$ for all $n > N$.
Then, combining \eqref{parseval-decomposed}, \eqref{finite-dimensional-bound} and \eqref{proj-bound} and using Parseval's identity \eqref{parseval-L2} again, we obtain
\begin{multline*}
\norm{u}_{\alpha,\gamma}^2
\leq C \left( \norm{S\sp{\alpha,\gamma}_0 u}_{\alpha,\gamma}^2 + \norm{\mathcal{D}\psp{\gamma} u}_{\alpha+1,\gamma}^2 \right)\\
+ \sup_{n > N} \frac{1}{n(n+2\lambda^{\alpha,\gamma})-\tilde C} \left[ \norm{\mathcal{D}\psp{\gamma} u}_{\alpha+1,\gamma}^2 + \sum_{1 \leq i < j \leq d} \norm{\mathcal{D}\psp{\gamma}_{i,j} u}_{\alpha,\gamma}^2 \right].
\end{multline*}
The result follows upon using the bounds $\norm{\cdot}_{\alpha+1,\gamma} \leq \norm{\cdot}_{\alpha,\gamma}$ and $\norm{\mathcal{D}\psp{\gamma}_{i,j} \cdot}_{\alpha,\gamma}^2 \leq 2 \norm{\mathcal{D}\psp{\gamma} \cdot}_{\alpha,\gamma}^2$.
\end{proof}
\end{proposition}

We can now prove our sharpness result.

\begin{thm}\label{thm:sharpness}
For all $\alpha > -1$ and $\gamma \in (-1,\infty)^d$, \autoref{thm:main} is sharp in the case $l=r=1$.
\begin{proof}
Let $P\sp{(\alpha,\beta)}_n$ denote the Jacobi polynomial of parameter $(\alpha,\beta)$ and degree $n$ \cite[Ch.~IV]{Szego:1975}.
From \cite[Eqs.~(4.21.7) and (4.3.3)]{Szego:1975} and \cite[Eq.~(6.4.21)]{AAR:1999},
\begin{equation}\label{jacobi-diff}
{P\sp{(\alpha,\beta)}_n}'(x) = \frac{n+\alpha+\beta+1}{2} P\sp{(\alpha+1,\beta+1)}_{n-1}(x),
\end{equation}
\begin{equation}\label{jacobi-norm}
\begin{aligned}
h\sp{(\alpha,\beta)}_n & := \int_{-1}^1 \abs{P\sp{(\alpha,\beta)}_n(x)}^2 (1-x)^\alpha (1+x)^\beta \dd x\\
& = \frac{2^{\alpha+\beta+1}}{2n+\alpha+\beta+1} \frac{\Gamma(n+\alpha+1)\Gamma(n+\beta+1)}{\Gamma(n+1)\Gamma(n+\alpha+\beta+1)},
\end{aligned}
\end{equation}
\begin{equation}\label{jacobi-idshift}
P\sp{(\alpha,\beta)}_n(x) = \frac{n+\alpha+\beta+1}{2n+\alpha+\beta+1} P\sp{(\alpha+1,\beta)}_n(x) - \frac{n+\beta}{2n+\alpha+\beta+1} P\sp{(\alpha+1,\beta)}_{n-1}(x);
\end{equation}
the last expression in \eqref{jacobi-norm} must be modified if $n=0$.
Let us adopt the abbreviation $s(\gamma) = \sum_{j=1}^d \gamma_j$.
Given $n \in \natural$, we define $t_{\alpha,\gamma,n} \in \poly^d_{2n}$ by
\begin{multline}\label{sharp-poly}
t_{\alpha,\gamma,n}(x) := \frac{2n+2\lambda^{\alpha,\gamma}-2}{4n+2\lambda^{\alpha,\gamma}-2} P\sp{(\alpha,\frac{1}{2}s(\gamma) + \frac{d-2}{2})}_n (2\norm{x}^2-1)\\
- \frac{2n+s(\gamma)+d-2}{4n+2\lambda^{\alpha,\gamma}-2} P\sp{(\alpha,\frac{1}{2}s(\gamma) + \frac{d-2}{2})}_{n-1} (2\norm{x}^2-1).
\end{multline}
From \cite[Prop.~8.1.5]{DunklXu:2014}, we learn that the first term defining $t_{\alpha,\gamma,n}$ in \eqref{sharp-poly} is a member of $\mathcal{V}\sp{\alpha,\gamma}_{2n}$ and the second is a member of $\mathcal{V}\sp{\alpha,\gamma}_{2n-2}$.
Therefore
\begin{multline}\label{sharp-residue}
R_{\alpha,\gamma,n}(x) := t_{\alpha,\gamma,n} - S\sp{\alpha,\gamma}_{2n-1}(t_{\alpha,\gamma,n})(x)\\
= \frac{2n+2\lambda^{\alpha,\gamma}-2}{4n+2\lambda^{\alpha,\gamma}-2} P\sp{(\alpha,\frac{1}{2}s(\gamma) + \frac{d-2}{2})}_n (2\norm{x}^2-1).
\end{multline}
As $R_{\alpha,\gamma,n}$ is a radial member of $\mathcal{V}\sp{\alpha,\gamma}_{2n}$, from part \ref{grad-estimation} of \autoref{approx-ineq},
\begin{equation}\label{sharp-residue-grad-norm}
\norm{\mathcal{D}\psp{\gamma} R_{\alpha,\gamma,n}}_{\alpha,\gamma}^2 = \frac{2n(2n+2\lambda^{\alpha,\gamma})(2n+\lambda^{\alpha,\gamma})}{\alpha+1} \norm{R_{\alpha,\gamma,n}}_{\alpha,\gamma}^2.
\end{equation}
Also,
\begin{equation}\label{sharp-residue-norm}
\begin{aligned}
\norm{R_{\alpha,\gamma,n}}_{\alpha,\gamma}^2 & = \frac{(2n+2\lambda^{\alpha,\gamma}-2)^2}{(4n+2\lambda^{\alpha,\gamma}-2)^2} \int_{B^d} \abs{P\sp{(\alpha,\frac{1}{2}s(\gamma) + \frac{d-2}{2})}_n (2\norm{x}^2-1)}^2 W_{\alpha,\gamma}(x) \dd x\\
& = \frac{(2n+2\lambda^{\alpha,\gamma}-2)^2}{(4n+2\lambda^{\alpha,\gamma}-2)^2} 2^{-(2+\alpha+\frac{1}{2}s(\gamma) + \frac{d-2}{2})} h_n\sp{\left( \alpha, \frac{1}{2} s(\gamma) + \frac{d-2}{2} \right)} \abs{\sphere^{d-1}}_\gamma,
\end{aligned}
\end{equation}
where $\abs{\sphere^{d-1}}_\gamma := \int_{\sphere^{d-1}} W_{0,\gamma}(x) \dd S(x)$; the integral was computed by first switching to generalized spherical coordinates and then performing the change of variable $t = 2r^2-1$.
Given $j \in \{1,\dotsc,d\}$,
\begin{align*}
\mathcal{D}\psp{\gamma}_j t_{\alpha,\gamma,n}(x) & \stackrel{\eqref{jacobi-diff}}{=} \frac{2n+2\lambda^{\alpha,\gamma}-2}{4n+2\lambda^{\alpha,\gamma}-2} x_j \Bigg[ (2n+2\lambda^{\alpha,\gamma} ) P\sp{(\alpha+1,\frac{1}{2}s(\gamma) +\frac{d}{2})}_{n-1}(2\norm{x}^2-1)\\
& \quad -(2n+s(\gamma) +d-2) P\sp{(\alpha+1,\frac{1}{2}s(\gamma) +\frac{d}{2})}_{n-2}(2\norm{x}^2-1) \Bigg]\\
& \stackrel{\eqref{jacobi-idshift}}{=} (2n+2\lambda^{\alpha,\gamma}-2) \, x_j \, P\sp{(\alpha,\frac{1}{2}s(\gamma) +\frac{d}{2})}_{n-1}(2\norm{x}^2-1).
\end{align*}
Hence,
\begin{equation}\label{sharp-poly-grad-norm}
\begin{split}
\norm{\mathcal{D}\psp{\gamma} t_{\alpha,\gamma,n}}_{\alpha,\gamma}^2 & = (2n+2\lambda^{\alpha,\gamma}-2)^2 \int_{B^d} \norm{x}^2 \abs{P\sp{(\alpha,\frac{1}{2}s(\gamma) +\frac{d}{2})}_{n-1}(2\norm{x}^2-1)}^2 W_{\alpha,\gamma}(x) \dd x\\
& = (2n+2\lambda^{\alpha,\gamma}-2)^2 2^{-(2+\alpha+\frac{1}{2} s(\gamma) + \frac{d}{2})} h_{n-1}\sp{(\alpha,\frac{1}{2} s(\gamma)+\frac{d}{2})} \abs{\sphere^{d-1}}_\gamma,
\end{split}
\end{equation}
where the integral over $B^d$ was computed similarly to that in \eqref{sharp-residue-norm}.
Therefore, for $n \geq 2$,
\begin{multline}\label{sharp-ratio-grad-norm}
\frac{\norm{\mathcal{D}\psp{\gamma} R_{\alpha,\gamma,n}}_{\alpha,\gamma}^2}{\norm{\mathcal{D}\psp{\gamma} t_{\alpha,\gamma,n}}_{\alpha,\gamma}^2} \stackrel{\eqref{sharp-residue-grad-norm}}{=} \frac{2n(2n+2\lambda^{\alpha,\gamma})(2n+\lambda^{\alpha,\gamma})}{\alpha+1} \frac{\norm{R_{\alpha,\gamma,n}}_{\alpha,\gamma}^2}{\norm{\mathcal{D}\psp{\gamma} t_{\alpha,\gamma,n}}_{\alpha,\gamma}^2}\\
\stackrel{\eqref{sharp-residue-norm},\eqref{sharp-poly-grad-norm}}{=} \frac{2n(2n+2\lambda^{\alpha,\gamma})(2n+\lambda^{\alpha,\gamma})}{\alpha+1} \frac{2 h_n\sp{\left( \alpha, \frac{1}{2} s(\gamma) + \frac{d-2}{2} \right)}}{(4n+2\lambda^{\alpha,\gamma}-2)^2 h_{n-1}\sp{(\alpha,\frac{1}{2} s(\gamma)+\frac{d}{2})}}\\
\stackrel{\eqref{jacobi-norm}}{=} \frac{4n(2n+2\lambda^{\alpha,\gamma})(2n+\lambda^{\alpha,\gamma})}{(\alpha+1)(4n+2\lambda^{\alpha,\gamma}-2)^2} \frac{(2n+\lambda^{\alpha,\gamma}-1) \Gamma(n+\alpha+1)\Gamma(n)}{2(2n+\lambda^{\alpha,\gamma})\Gamma(n+\alpha)\Gamma(n+1)}\\
= \frac{(2n+2\lambda^{\alpha,\gamma})(n+\alpha)}{(\alpha+1)(4n+2\lambda^{\alpha,\gamma}-2)} \sim \frac{2n-1}{4(\alpha+1)} \quad \text{as $n \to \infty$},
\end{multline}
where we have exploited the identity $\Gamma(z+1) = z \Gamma(z)$ and we use $\sim$ to denote that the ratio of two expressions thus linked tends to $1$.
As $u \mapsto \norm{\mathcal{D}\psp{\gamma} u}_{\alpha,\gamma} + \norm{S\sp{\alpha,\gamma}_0(u)}_{\alpha,\gamma}$ is an equivalent norm for $\HH^1_{\alpha,\gamma}$ (cf.~\autoref{pro:H1-equivalence}) and both $t_{\alpha,\gamma,n}$ and $R_{\alpha,\gamma,n}$ are $\LL^2_{\alpha,\gamma}$-orthogonal to $\mathcal{V}\sp{\alpha,\gamma}_0$ if $n \geq 2$, we infer from \eqref{sharp-ratio-grad-norm} that there exists a positive constant $C$ depending on $d$, $\alpha$ and $\gamma$ only such that
\begin{equation*}
\lim_{n\to\infty} \frac{\norm{t_{\alpha,\gamma,n} - S\sp{\alpha,\gamma}_{2n-1}(t_{\alpha,\gamma,n})}_{\alpha,\gamma;1}}{\norm{t_{\alpha,\gamma,n}}_{\alpha,\gamma;1} (2n-1)^{1/2}} = C.
\end{equation*}
Thus, the $l=r=1$ instance of \autoref{thm:main} is sharp, because otherwise the left-hand side limit would vanish.
\end{proof}
\end{thm}

\subsection*{Conclusion}

We have proved our mismatched approximation result \autoref{thm:main} and its sharpness for special values of the regularity parameters of the function being approximated and the norm used to measure the error.
On the way, we developed a suite of auxiliary results connecting Dunkl operators and $\LL^2_{\alpha,\gamma}$-orthogonal polynomials.

Starting from this work, some avenues of further work that we detect are:
(1) Adapting our arguments to weights invariant with respect to other reflection groups.
(2) Find analogues of Dunkl operators that raise or lower components of $\gamma$ instead of $\alpha$.
(3) Explore how \autoref{thm:main} fares under polynomial-preserving mappings to other domains.

\bibliographystyle{amsplain}
\bibliography{opriw-refs}

\end{document}